\documentclass{amsart}
\usepackage{amsthm,amsmath,amssymb,amscd,graphics,epsfig,epic,eepic}

\usepackage[]{graphicx}
 \usepackage[all]{xy}

 \newcommand {\Star} {{\mathrm Star}}

 \newcommand {\conv} {{\mathrm conv}}

 \newcommand {\Z} {{\mathbb Z}}

 \newtheorem{lemma}[subsection]{Lemma}

\newcommand{\entrylabel}[1]{\mbox{\textsf{{\rm c}1}}\hfil}
{\end{list}}
{
   \newtheorem{theorem}{Theorem}[subsection]
   \newtheorem{proposition}[theorem]{Proposition}
   \newtheorem{corollary}[theorem]{Corollary}

}
{\theoremstyle{definition}
   
   \newtheorem{example}[theorem]{Example}
   \newtheorem{definition}[theorem]{Definition}
}
{\theoremstyle{remark}
   \newtheorem*{remark}{Remark}
   
}

\newcommand{\QQ}{{\mathbb{Q}}}

\newcommand{\CC}{{\mathbb{C}}}
\newcommand{\bP}{{\mathbb{P}}}
\newcommand{\bA}{{\mathbf{A}}}

\newcommand{\cI}{{\mathcal I}}

\newcommand{\cO}{{\mathcal O}}

\newcommand{\res}{{\operatorname{res}}}

\newcommand{\inte}{{\operatorname{int}}}

\newcommand{\Link}{\operatorname{Link}}

\newcommand{\Spec}{\operatorname{Spec}}

\newcommand{\dar}{\downarrow}
\newcommand{\uar}{\uparrow}

\newcommand{\Reg}{{\operatorname{Reg}}}

\begin{document}


\title[Equisingular resolution with SNC fibers and combinatorial type of varieties ]{Equisingular resolution with SNC fibers and \\ combinatorial type of varieties}

\author{ {Jaros{\l}aw W{\l}odarczyk}}

\thanks{W\l odarczyk was supported in part by NSF grant DMS-0100598 and  and BSF grant 2014365}


\address{J. Wlodarczyk, Department of Mathematics\\Purdue University\\West
Lafayette, IN-47907\\USA}
\email{wlodarcz@purdue.edu, wlodar@math.purdue.edu}

\date{\today}
\begin{abstract} 
We introduce the notion of combinatorial type of  varieties $X$ which generalizes the concept of the dual complex of SNC divisors. It is a unique, up to homotopy, finite simplicial complex $\Sigma(X)$ which is  functorial with respect to  morphisms of varieties.
Its cohomology $H^i(\Sigma(X),Q)$ for complex projective varieties coincide with  weight zero part of the Deligne filtration $W_0(H^i(X,Q))$.
The notion can be understood as a topological measure of the singularities of algebaric schemes of finite type.

We  also prove  that any variety in characteristic zero admits the Hironaka desingularization with all fibers having SNC. Moreover the dual complexes of the fibers are isomorphic on strata.
Also for any morphism $f:X\to Y$ there exists a similar  desingularization $\tilde{X}\to X$ for which the induce morphism $\tilde{X}\to Y$ has SNC fibers.

 One of the consequence is that for any projective morphism $f:X\to Y$ the combinatorial type  of the fiber is a constructible function. In particular $\dim(W_0H^i(f^{-1}(y))$ is constructible.

\end{abstract}
\maketitle

\bigskip

\tableofcontents
\addtocounter{section}{-1}

\section{Introduction} 

One of the purpose of the paper is to extend the Hironaka desingularization theorems in order to control the fibers of resolutions and more generally morphisms between varieties. In particular if $D$ is a SNC  divisor
associated to a resolution of an  isolated singularity then  it was obseved independently by
Kontsevich, Soibelman \cite[A.4]{ks}, 
Stepanov \cite{step2}, and others  that the homotopy type of the
 dual complex associated with $D$ is an invariant
for the singularity (\cite{ks}). More generally   Thuillier \cite{th} and 
Payne \cite{payne} proved  homotopy  invariance results for  boundary divisors.  
In characteristic zero, all these results can be derived from   the weak factorization theorem of W\l
odarczyk [Wlo], and Abramovich-Karu-Matsuki-W\l odarczyk
[AKMW]. A refinement of factorization theorem gives a more general version of the homotopy invariance for varieties with SNC, and allows to study  lower dimensional fibers of  morphisms. 
In particular, in the paper \cite{abw} we show that the invariants of the dual complexes of the fibers of the resolution can be defined for arbitrary varieties. Moreover they can be computed directly when using fibers with SNC crossings of smaller dimension. This raises a question of existence of  resolution with SNC fibers. In this paper we show that such a resolution can be constructed. Moreover one can associate with it   a stratification, with points in the strata defining isomorphic dual complexes.
In this sense it can be considered as a version of equisingular resolution. On the other hand if one considers an arbitrary morphism $f:X\to Y$ the method allows to desingularize the variety $X$ and the fibers of the morphism transforming them into SNC varieties. Since, in general, we cannot eliminate SNC singularities of the fibers  the theorem can be considered as the desingularization of the fibers of the morphism. The proofs  rely on Hironaka desingularization thorems and, introduced here, notion of SNC morphism  generalizing smooth morphisms.

The considerations of the dual complexes associated with SNC fibers of desingularizations naturally lead to the idea of extension of the concept of the dual complex and placing it in a more general context. We introduce here the notion of {\it combinatorial type} which allows to approach  the theory of singularities from the topological perspective. The combinatorial type $\Sigma(X)$ is a roughly a homotopy type of topological space (simplicial complex) which is assigned to an algebraic scheme $X$  and which reflects the singularities of the scheme. Moreover the notion, in fact defines, a functor from the category of (quasiprojective) algebraic schemes of finite type to the homotopy category of topological spaces. In the particular case of  SNC divisors the combinatorial type is the homotopy type of its dual complex. On the other hand, in the case of complex projective varieties, the rational cohomology $H^i(\Sigma(X),\QQ)$ of the combinatorial type coincides with weight zero part of the Deligne filtration $W_0H^i(X,\QQ)$.

The paper is organized as follows. In the first  section we briefly formulate the Hironaka desingularization theorems in their stronger form which  are used frequently in the remaining part  of the paper
. In Section 2 we briefly discuss  the language of varieties  with SNC and and associated simplicial complexes. One of the main tool used in the considerations is , introduced here, operation of cone extension. It generalizes the operation of star subdivision and can be conveniently used for natural modification of dual complexes  associated with (embedded) blow-ups of SNC varieties at compatible SNC centers. (One should mention that the case of varieties with SNC differs quite significantly from the case of SNC divisors.)
 In
Section 3 we introduce the notion of combinatorial type of varieties and associated morphisms between them, and study their basic properties.
In Section 4 we intoduce the notion of SNC morphism and prove the analog of generic smoothness theorem for morphisms of SNC varieties. In section 5 we study \'etale trivialization  of SNC morphisms. In section 6 we prove the desingularization theorems with SNC fibers, and some of their consequences. In particular we show that the combinatorial type of the fibers of a projective morphism is constant on  strata.
We note that in this paper the term ``scheme'' refers to a  scheme of  finite type over a ground field $K$ of characteristic zero. A
 variety is a reduced scheme.

\section{Formulation of the Hironaka  resolution  theorems}

In the paper we are going to use and generalize the following Hironaka desingularization theorems in the strongest form. 
(see \cite{Hir},\cite{BM2},\cite{Vi},\cite{Wlo3}). Observe that in the actual Hironaka algorithm we have no control on the fibers of the morphisms and it seems rather difficult to alter the steps of the algorithm to ensure the SNC condition on the fibers.
Throughout the paper we shall work over a ground field $K$ of characteristic zero. \begin{enumerate}

\item{\bf Strong Hironaka Resolution of Singularities}
\begin{theorem} \label{th: 3} Let $Y$ be an algebraic variety over a field of characteristic zero.

There exists a canonical desingularization of $Y$ that is
a smooth variety $\widetilde{Y}$ together with a projective birational morphism $\res_Y: \widetilde{Y}\to Y$ such that
\begin{enumerate}
\item  $\res_Y$ is a composition  of blow ups $Y=Y_0\leftarrow Y_1\leftarrow\ldots\leftarrow \widetilde{Y}$ with smooth centers disjoint from the set of regular points  $Reg(Y)$ of $Y$.

\item $\res_Y$ is functorial with respect to smooth morphisms.
For any smooth morphism $\phi: Y'\to Y$ there is a natural lifting $\widetilde{\phi}: \widetilde{Y'}\to\widetilde{Y}$ which is a smooth morphism. 

Moreover the centers of blow ups are defined by the liftings of the centers of blow-ups

\item $\res_Y: \widetilde{Y}\to Y$ is  an isomorphism over the nonsingular part of $Y$. Moreover $\res_Y$ is equivariant with respect to any group action not necessarily preserving the ground field.
\end{enumerate}
\end{theorem}

\item {\bf  Strong Hironaka Embedded  Desingularization} 

\begin{theorem} \label{th: emde} \label{th: 2} Let $Y$ be a subvariety of a smooth variety
$X$ over a field of characteristic zero.
There exists a  sequence $$ X_0=X \buildrel \sigma_1 \over\longleftarrow X_1
\buildrel \sigma_2 \over\longleftarrow X_2\longleftarrow\ldots
\longleftarrow X_i \longleftarrow\ldots \longleftarrow X_r=\widetilde{X}$$ of
blow-ups  $\sigma_i:X_{i-1}\longleftarrow X_{i}$ of smooth centers $C_{i-1}\subset
 X_{i-1}$ such that

\begin{enumerate}

\item The exceptional divisor $E_i$ of the induced morphism $\sigma^i=\sigma_1\circ \ldots\circ\sigma_i:X_i\to X$ has only  simple normal
crossings and $C_i$ has simple normal crossings with $E_i$.

\item Let $Y_i\subset X_i$ be the strict transform of $Y$. All centers $C_i$  are contained in $Y_i$, and are disjoint from the regular point set $\Reg(Y)\subset Y_i$ of points where   $Y_i$   is smooth.

\item  The strict transform $\widetilde{Y}:=Y_r$ of  $Y$  is smooth and
has only simple normal crossings with the exceptional divisor $E_r$.

\item The morphism $(X,{Y})\leftarrow (\widetilde{X},\widetilde{Y})$ defined by the embedded desingularization commutes with smooth morphisms and  embeddings of ambient varieties. It is equivariant with respect to any group action not necessarily preserving the ground $K$.

\end{enumerate}
\end{theorem}

\item {\bf Canonical Principalization }
\begin{theorem} \label{th: 1} Let ${\cI}$ be a sheaf of ideals on a smooth algebraic variety $X$, and  $Y\subset X$ be any subvariety of $X$.
There exists a principalization of ${\cI}$  that is, a sequence

$$ X=X_0 \buildrel \sigma_1 \over\longleftarrow X_1
\buildrel \sigma_2 \over\longleftarrow X_2\longleftarrow\ldots
\longleftarrow X_i \longleftarrow\ldots \longleftarrow X_r =\widetilde{X}$$

of blow-ups $\sigma_i:X_{i-1}\leftarrow X_{i}$ of smooth centers $C_{i-1}\subset
 X_{i-1}$
such that

\begin{enumerate}

\item The exceptional divisor $E_i$ of the induced morphism $\sigma^i=\sigma_1\circ \ldots\circ\sigma_i:X_i\to X$ has only  simple normal
crossings and $C_i$ has simple normal crossings with $E_i$.

\item The
total transform $\sigma^{r*}({\cI})$ is the ideal of a simple normal
crossing divisor
$\widetilde{E}$ which is  a natural  combination of the irreducible components of the divisor ${E_r}$.

\end{enumerate}
The morphism $(\widetilde{X},\widetilde{\cI})\rightarrow(X,{\cI}) $ defined by the above principalization  commutes with smooth morphisms and  embeddings of ambient varieties. It is equivariant with respect to any group action not necessarily preserving the ground field $K$.

\end{theorem}

\item {\bf  Extended Hironaka Embedded  Desingularization}

In the constructions used in this paper it will be  convenient to use the following modification of the strong Hironaka Embedded desingularization.
\begin{definition} Let $Y$ be a subvariety of a smooth variety
$X$ over a field of characteristic zero.

By the {\it extended Hironaka desingularization} we shall mean a 
sequence of blow-ups as in Hironaka embedded desingularization followed by the blow-up $\overline{X}\to \widetilde{X}$ of the smooth center of the strict transform $C=\widetilde{Y}$ on $\widetilde{X}$.
\end{definition}

Observe that extended  Hironaka  desingularization creates SNC divisor $D$ on $\overline{X}$.
It is functorial, commutes with embeddings of ambient varieties.
Moreover all the centers of the blow-us are contained in the strict transforms of $Y$, and in particular are of dimension $\leq \dim(Y)$. 

\end{enumerate}

\section{Dual complexes of SNC varieties and cone extensions}

\begin{definition} \label{de: fan} By an {\it abstract polyhedral complex}
$\Sigma $  we mean a finite partially ordered $\leq$ collection of polytopes $\sigma$  together with inclusion maps $i_{\tau\sigma}:\tau\to \sigma$ for any $\tau\leq \sigma$ 
that 
\begin{enumerate}
\item For any $\tau\leq \sigma$, the subset $i_{\tau\sigma}(\tau)$ is a face of $\sigma$ 
\item For any face $\sigma'$ of a polytope $\sigma\in\Sigma$, there exists a $\tau\in \Sigma$ such that $i_{\tau\sigma}(\tau)=\sigma'$.

\item For any $\tau,\tau'\leq \sigma$ such that 
$i_{\tau\sigma}(\tau)=i_{\tau'\sigma}(\tau')$ we have that
$\tau=\tau'$.

\end{enumerate} 

If all faces are simplices then the complex will be called {\it simplicial}. 
A subset $\Sigma'$ of a polyhdral complex which is a polyhedral complex  itself will be called a subcomplex 
If $\tau\leq \sigma$ then we shall also call $\tau$ a {\it face of $\sigma$ } slightly abusing terminology.

By the {\it geometric realization} $|\Sigma|$ of $\Sigma$ we mean the
topological space $$\coprod_{\sigma\in \Sigma} \,\, \sigma \,\,\slash\sim\,\,,$$
when $\sim$ is the equivalence relation generated by the inclusion maps $i_{\tau\sigma}$

\end{definition}

Note that we do not require that polytopes intersect along faces but rather subcoplexes.

The faces of polytopes $\sigma$ define unique elements of $\Sigma$, and thus $|\Sigma|$ is obtained by gluing maximal polytopes along their faces.

\begin{definition} The {\it map of polyhedral complexes} $\phi: \Sigma\to \Sigma'$ is a continuous map of the geometic realizations $|\phi|: |\Sigma|\to |\Sigma'|$, and the compatible affine maps of the faces
$\phi_\sigma: \sigma\to \sigma'$, where $\sigma'$ is the smallest face  $\sigma'\in \Sigma'$ such that $|\phi|(\sigma)\subset \sigma'$, and the following diagram commutes:$$\begin{array}{cccl}

\sigma &\buildrel{\phi_\sigma}\over\to  & \sigma'&  \\
\cap  && \cap &\\
 |\Sigma|&\buildrel{|\phi|}\over\to &|\Sigma'| &\end{array}$$

 If $|\phi|$ is a homeomorphism then we shall call $\Sigma$ {\it a subdivision} of $\Sigma'$. If, moreover $\Sigma$ is simplicial then it will be called {\it triangulation} of $\Sigma'$.

\end{definition}

In the paper we shall consider mostly simplicial complexes referring to them as {\it complexes}. The polyhedral complexes  occur only when introducing cartesian products.

\begin{definition} A {\it variety with simple normal crossings (SNC)} is a reduced sheme $X$ of finite type with all maximal components $X_i$ smooth over $K$, and having SNC crossings. That is for any $p\in X$ there is a neighborhood $U$ of $p$ in $X$ , and 
an \'etale morphism of $U\to Z$, where $Z\subset \bA^n$ is a union of coordinate subspaces of possibly different dimension.
If $X$ is an SNC  variety on  a smooth ambient variety $Z$, and $C$ is a smooth subvariety then we say that $C$ has SNC with $X$ if  for any $p\in X$ there is a neighborhood $U$ of $p$ in $Z$, and 
an \'etale morphism of $\phi: U\to \bA^n$, where $C$ is the preimage of a coordinate subspace and $X$ it the preimage    of a union of coordinate subspaces of possibly different dimension.

\end{definition}

Recall that given an SNC divisor or more generally a variety with SNC  $D=\bigcup D_i$ with maximal components $D_i$  one can associates with it the dual complex.   
 More precisely  we define the dual complex $\Delta(D)$ to be the simplicial complex whose vertices correspond to the maximal components $D_i$, and 
  simplices $\Delta_{j_0\ldots,j_p}^s$ or shortly $\Delta_\alpha$, where $\alpha:=(j_0\ldots,j_p:s)$ are in correspondence  to the irreducible $D_\alpha$ components of $(p+1)$-fold nonempty intersections $D_{j_0}\cap\ldots\cap D_{j_p}$.  
Note that the components $D_\alpha$ need not to be distinct for different $\alpha$.
  
 The dual complexes are often constructed from SNC divisors which are obtained by succesive blow-ups of smooth centers having SNC with the exceptional divisors.
It is well known fact that when blowing up a center which is an intersection component of the divisor we form a new complex which is the star subdivision at the simplex corresponding to the component. In general one needs to extend the notion of star subdivision to the, introduced here more universal transformation of {\it cone extension} (Definition \ref{cone}, Porposition \ref{cone2}).

\bigskip

For any two faces $\sigma$ and $\tau$ in a complex $\Sigma$ we write 
$\sigma<\tau$ if $\sigma$ is a proper face of $\tau$.

For any complex $\Sigma$ and its subset $\Delta\subset$ denote by  {\it closure} $\overline{\Delta}=\{\tau\mid \tau \leq \sigma, \sigma\in \Delta\}$.

\begin{definition}\label{de: star} Let $\Sigma$ be a complex 
and $\tau \in \Sigma$ be its face. The {\it star} of the
cone $\tau$, the {\it closed star}, and the {\it link} of $\tau$ are
defined as follows:
 $${\rm Star}(\tau ,\Sigma):=\{\sigma \in \Sigma\mid 
\tau\leq \sigma\},$$ 
$$\overline{{\rm Star}}(\tau ,\Sigma):=\{\sigma \in
\Sigma\mid  \sigma'\leq \sigma  \mbox{ for some }  \sigma'\in
{\rm Star}(\tau ,\Sigma)\}.$$ 
$$\Link(\tau ,\Sigma):=\overline{{\rm Star}}(\tau ,\Sigma)\setminus {\rm Star}(\tau ,\Sigma)$$
\end{definition} 

We say that a  subset $\Delta$ of a complex  $\Sigma$ is {\it star-closed} iff
whenever $\tau\leq\sigma$  and $\tau\in \Delta$, and $\sigma\in \Sigma$ we have that $\sigma\in \Delta$.

 By the {\it support} of a subset $\Delta$ of a complex  $\Sigma$ we mean the union of the relative interior of all
its faces, 
$|\Delta|=\bigcup_{\tau\in \Delta}\inte(\tau)\subset|\Sigma|$. Observe that the support of the complex $\Sigma$ coincides with its geometric realization $|\Sigma|$. That is why we use the same notation for both notions. In general 
$|\Delta|$ is a locally closed subset of $|\Sigma|$.

The definition of the support gives a  convenient description of topology on $|\Sigma|$.

 \begin{lemma} Let $\Delta$ be a subset of $\Sigma$. Then
 \begin{enumerate}
  \item The subset $\Delta$ is star closed in $\Sigma$ iff its complement $\Sigma\setminus \Delta$ is a subcomplex of $\Sigma$.

 \item The subset $\Delta$ is star closed in $\Sigma$ iff $|\Delta|$ is open in $|\Sigma|$
 \item $\Delta$ is a subcompex of $\Sigma$ iff $|\Delta|$ is closed in $|\Sigma|$.
\item $|\overline{\Delta}|$ is the closure of $|\Delta|$. 
 \end{enumerate}
 \end{lemma}
\begin{remark} Note that the components $D_\alpha$ may correspond to different simplices. For any such a component $C=D_\alpha$ there exists a unique maximal face $\sigma_C$ corresponding to all the divisors $D_i$ containing  $C$. Then all other simplices $\tau$ with the property $D_\tau=C$ correspond to certain faces of $\sigma_C$.
\end{remark}

\begin{remark} If $\tau \leq \tau'$ then $D_{\tau'}\subset D_{\tau}$. On the other hand if $D_{\tau'}\subset D_{\tau}$ then
$D_{\tau'}\cap D_{\tau}\subset D_{\tau}$ which means that there exists a maximal  simplex $\tau'_{max}$ such that $D_{\tau'}=D_{\tau'_{max}}$ and for which $\tau'\leq {\tau'_{max}}$, and $\tau\leq {\tau'_{max}}$. Thus for $\tau'=\tau'_{max}$, we get that 
$D_\tau \supseteq D_{\tau'_{max}}$ iff $\tau\leq \tau'_{max}$.
\end{remark}

\begin{definition} By the {\it interior} of a subset $\Delta$ in the complex $\Sigma$ we mean the maximal star closed subset $\inte(\Delta)$ of $\Delta$
\end{definition}

If $v$ is independent of  cones $\sigma\in \Delta$ then by $v*\Delta(D)$ we shall denote the complex consisting of the cones $\conv(v,\sigma)$  
over $\sigma\in \Delta$ with vertex $v$. We also put  $v*\emptyset=\{v\}$. 

\begin{definition}\label{de: star subdivision} Let $\Sigma$ be a complex and
$v$ be a point in the
relative interior of $\tau\in\Sigma$. Then the {\it star
subdivision}  of $\Sigma$ with respect to
$\tau$ is defined to be
$$v\cdot\Sigma=(\Sigma\setminus {\rm Star}(\tau ,\Sigma) )\, \cup\,
v*(\overline{\rm Star}(\tau
,\Sigma)\setminus {\rm Star}(\tau
,\Sigma)).$$ \end{definition}

Note that in the definition $v$ can be chosen as any vertex independent of the faces in the $\Link(\tau
,\Sigma)$. The assumption that $v$ is in the
relative interior of $\tau$ is to visualize the operation and can be dropped.


\begin{definition}\label{cone} Given a complex $\Sigma$, and  its subset $\Delta^0$ which is star closed and contained in the subcomplex $\Delta$ of $\Sigma$ we define a {\it cone extension} of $\Sigma$ at $(\Delta,\Delta^0)$ to be  the complex $$(\Delta,\Delta^0)\cdot \Sigma:=\Sigma\setminus \Delta^0 \cup v*(\Delta\setminus \Delta^0),$$
where $v$ is any new vertex which is independent from faces in $\Delta\setminus \Delta^0$. By the {\it pure cone extension} of  $\Sigma$ we mean the complex $(\Delta,\emptyset)\cdot \Sigma$, where $\Delta$ is a subcomplex of $\Sigma$
\end{definition}

Note that $\Delta^1:=\Delta\setminus \Delta^0$ is necessarily a subcomplex of $\Sigma$ and $\Delta$.  On the other hand $\Delta^0\subset \inte(\Delta)$.
\begin{example} Let $\tau\in \Sigma$ be arbitrary face. Consider the subsets $$\Delta^0={\Star(\sigma,\Sigma)},\,\,\,\,\Delta=\overline{\Star(\sigma,\Sigma)}.$$ In that case $$\Delta^1=\overline{\Star(\sigma,\Sigma)}\setminus {\Star(\sigma,\Sigma)}=\Link(\sigma,\Sigma),$$ and 
the cone extension of $\Sigma$ at $(\Delta,\Delta^0)=(\overline{\Star(\sigma,\Sigma)},{\Star(\sigma,\Sigma)})$ is the star subdivision of $\Sigma$ at $\sigma$.
\end{example}

\begin{example} The cone extension of  $\Sigma$ at $(\emptyset,\emptyset)$ is the disjoint union $\Sigma\cup \{v\}$, and thus the operation changes the homotopy type of $|\Sigma|$.

\end{example}

 For any variety $D$ with SNC denote by $D_\tau$, where $\tau\in \Delta(D)$, the  stratum which is the intersection of divisors corresponding to vertices of $\tau$, and by $v_C$

 \begin{proposition}\label{cone2} Let $C$ be a center having SNC with a variety with SNC $D$ on  a smooth variety $X$. Set $$\Delta^0_C:=\{\tau\in \Delta(D) \mid D_\tau\subset  C \} $$
 $$\Delta_C:=\{\tau\in \Delta(D) \mid D_\tau\cap C\neq \emptyset\} $$

Then $\Delta^0_C$ is star closed, and
the new complex $\Delta(D')$ arising from $D$ after blow-up at $C$ is given by the cone extension of $\Delta(D)$ at $(\Delta_C,\Delta^0_C)$.
$$\Delta(D')=(\Delta_C,\Delta^0_C)\cdot \Delta(D)$$
Moreover  \begin{enumerate}
\item The  new vertex $v=v_C$ corresponds to the exceptional divisor $E$.
 \item The components which are eliminated after blow-up are those corresponding to $\Delta^0_C$.
\item The components created in the blow-up coorespond to faces  in \,\, $(v_C*\Delta^1_C)\,\, \setminus \,\,  \Delta^1_C$.
 
 \item  If the center  $C$ is the intersection  component and corresponds to the face $\sigma_C$ then $$\Star(\sigma_C,\Delta(D))\subseteq\Delta^0_C\,\,\subseteq \Delta_C=\overline{\Star(\sigma_C,\Delta(D))}.$$
 In this  case   $|\Delta(D)|$ and $|\Delta(D')|$ are homotopy equivalent.

\item If the center  $C$ is properly contained in the smallest intersection  component containing it then $$\Delta_C\subseteq\overline{\Star(\sigma_C,\Delta(D))}$$ with the maximal simplices of $\Delta_{C}$  in $\Star(\sigma_{C},\Delta(D)$.
In this case  $\Delta_C$ is contractible and
 $|\Delta(D')|$  retracts homotopically to  $|\Delta(D)|$.  
 
 \end{enumerate}
 \end{proposition}
\begin{proof}
 (1), (2), and (3) follow from the definition and properties of blow-ups.

 (4) Assume  that $C=D_{\sigma_C}$ be the minimal component describing $C$.
Then $\tau\in \Delta_C$ then  $D_\tau\cap C \neq \emptyset$,
 and $D_\tau \cap D_{\sigma_C} \neq \emptyset$. In other words $\Star(\tau,\Delta(D))\cap \Star(\sigma_C,\Delta(D))\neq \emptyset$, and  
 the maximal simplices of $\Delta_{C}$ are in $\Star(\sigma_{C},\Sigma_D)$. Also  if $\tau\in\Star(\sigma_C,\Delta(D))$
 then $D_\tau\subset {D_{\sigma_C}}=C$, which means that $\tau \in \Delta_C$. Thus $\Star(\sigma_C,\Delta(D))\subset 
 \Delta_C$.
 
  Consider the pure cone extension $(\Delta(D))''=\Delta(D) \cup v*(\Delta_C)$ of $\Delta(D)$ at $\Delta_C=\overline{\Star(\sigma_C,\Delta(D))}.$ Let $v_C$ be the barycenter of $\sigma_C$. Any maximal face $\tau$ in $\Delta^0_C$ is star closed  in $\Delta(D)$.
Consider its barycenter $v_\tau$. Then moving $v_\tau$ to $v$ defines the retraction eliminating face $\tau$. Then we eliminate the next maximal face $\tau_2$ in $\Delta^0_C\setminus \setminus \{\tau\}$. The process will continue upon eliminating all the faces in $\Delta^0_C$ and transforming homotopically 
 $|(\Delta(D))''|$ into $|(\Delta(D))'|=|\Delta(D)\setminus \Delta^0_C \cup v*(\Delta_C\setminus \Delta^0_C)|$.
 
  Note also that both $\Delta_C=\overline{\Star(\sigma_C,\Delta(D))}$, and $v*(\Delta_C)$ are contractible. Moving $v$ to $v_\tau$ defines retraction of $(\Delta(D))''=\Delta(D) \cup v*(\Delta_C)$ to its subcomplex  $\Delta(D)$.
 
 (5) Assume $C$ is properly contained in  $D_{\sigma_C}$.
Then  as in (4)  if $\tau\in \Delta_C$ then  $D_\tau\cap C \neq \emptyset$ and $D_\tau \cap D_{\sigma_C} \neq \emptyset$. As before  $\Star(\tau,\Delta(D))\cap \Star(\sigma_C,\Delta(D))\neq \emptyset$, and  
 the maximal simplices of $\Delta_{C}$ are in $\Star(\sigma_{C},\Sigma_D)$. 
 

Write $\Delta(D')=\Delta(D)\setminus \Delta_C^0 \cup v*(\Delta_C\setminus \Delta_C^0)$. Since ${\Star(\sigma_C,\Delta(D))}\subset \Delta^0_C$ the complex $\Delta(D')$ this is a subcomplex of
  the star subdivision $\Delta(D)\setminus {\Star(\sigma_C,\Delta(D))} \cup v*(\Delta_C\setminus {\Star(\sigma_C,\Delta(D))})$.
As before we eliminate one by one maximal faces $\tau$ in $\Delta^0_C$   in $\Delta(D)\setminus {\Star(\sigma_C,\Delta(D))}$ by moving their  barycenter $v_\tau$  to $v$ defines the retraction eliminating face $\tau$. The process will transform $\sigma\cdot \Delta(D)$ to $\Delta(D')$.

\end{proof}
\begin{remark} Assume  $C$ is the component of $D$ that corresponds to certain maximal face $\sigma_C\in \Delta(D)$ with the property $C=D_{\sigma_C}$. Then the set  $\Sigma_C:=\{\sigma\in \Delta(D)\mid D_\sigma=C\}$ consists $\sigma_C$ and some of its faces.
 Then 
if $D_\tau\subset C$ then $\sigma_C\leq \tau_{max}$ which means $\tau\in  \Star(\sigma,\Delta(D))$ for $\sigma \in \Sigma_C$. The component $D_\tau$ is contained in  $C=D_\tau\subset \overline{D_{\sigma_C}}$ iff $\tau\in\Star(\sigma_C,\Delta(D))$.
 
Since the set   $\Delta^0_C=\{\tau \in \Delta(D)\mid D_\tau\subset C\}$ is contained in $\overline{\Star(\sigma_C,\Delta(D))}$, it is 
equal to $$\Delta^0_C= \bigcup_{\sigma \in\Sigma_C}\Star(\sigma,\overline{\Star(\sigma_C,\Delta(D))})$$

\end{remark}
\begin{remark} Consider a union of the coordinate axes in $\bA^3$.
It is an SNC variety $E$with maximal components $E_i$ ,$i=0,1,2$ given by the axes. The dual complex  $\Delta(E)$  consists of the simplex $\Delta(e_0,e_1,e_2)$ and their faces. The cone $\Delta(e_0,e_1,e_2)$ and its all $1$ -dimensional faces  correspon the origin $O$, and they form the subset $\Delta^0_C$ for $C=\{0\}$. The set $\Delta_C$ coincides with $\Delta(E)$.
  The cone extension of $\Delta(E)$ at $(\Delta_C,\Delta^0_C)$ corresponding to the blow-up of $C=\{0\}$ is given by $v*(\Delta(E)\setminus \Delta^0_C
$, and thus consists of three one dimensional simplices $\Delta(v,e_i)$ interscecting at $v$. It corresponds to the exceptional divisors meeting the three disjoint strict transforms of the axes.

\end{remark}

The particular case of SNC divisor is somewhat simpler and was first studied by
Stepanov \cite{step2} (Cases (1),(2), and (4) )
\begin{proposition}\label{cone3} Let $C$ be center having SNC with a  SNC divisor $D$ on  a smooth variety $X$. 

Then the following possibilities hold

  \begin{enumerate}
\item  If the center  $C$ is the intersection  component and corresponds to the face $\sigma_C$ then $\Delta^0_C=\Star(\sigma_C,\Delta(D))$, and $\Delta_C=\overline{\Star(\sigma_C,\Delta(D))}$.
 Consequently the new  complex $\Delta(D')$ is given by the star subdivision of $\Delta(D)$ at $\sigma_C$.
 
\item If the center  $C$ is properly contained in the smallest intersection  component corresponding to $\sigma_C$  then $\Delta^0_C=\emptyset$, $\Delta_C\subset\overline{\Star(\sigma_C,\Delta(D))}$. Moreover the maximal simplices of $\Delta_C$ are in $\Star(\sigma_C,\Delta(D))$.

In this case $\Delta(D')$ is a pure cone extension at $(\overline{\Star(\sigma_C,\Delta(D))},\emptyset)$, and $\Delta(D')$ is homotopy equivalent to $\Delta(D)$.

\item If $C$ is not contained in $D$, and $D$ is an SNC divisor then
$D$ contains no components of $D$ and thus $\Delta^0_C=\emptyset$, and $\Delta(D')$ is a pure cone extension of $\Delta(D)$ at $(\Delta_C,\emptyset)$.
\item In the case (3) if $D''$ is the inverse image of $D$ then
$\Sigma(D)=\Sigma(D')$.
 
 \end{enumerate}
 \end{proposition}
\begin{proof}  Suppose $C$ contains a component $D_\alpha$ which is locally desccribed by the compatible coordinates $u_1=\ldots=u_k=0$, where $u_i$ describe the divisorial components. Then the center must have a form $u_1=\ldots=u_r=0$, for $r\leq k$, and appropriate coordinate  rearrangements. This means that it is a component of $D$. Then $\Delta^0_C=\emptyset$ in case (2) and (3). The rest of the assertions follows from the previous Proposition.

\end{proof}


\begin{lemma}\label{impo} Consider any embedding $ Z\subset T$ of smooth varieties. Let $\overline{Z}\to Z$ be a sequence of blow ups of nonsingular centers
having  SNC with exceptional divisors and let $\overline{T}\to T$ be the induced morphism. Denote the exceptional (SNC) divisor of $\overline{Z}\to Z$ by $D^Z$,  and by $D^T$ be the exceptional divisor of $\overline{T}\to T$. Then 
\begin{enumerate}
\item $D^Z=D^T\cap \overline{Z}$
\item  There is an inclusion of complexes $\Delta_Z:=\Delta(D^Z)\subset \Delta_T:=\Delta(D^T)$.
\item  Each intersection component $D^Z_\alpha$ is contained in a unique minimal component $D^T_\alpha$. Moreover the components $D^Z_\alpha$ are in bijective correspondence with the components of $D^T_\alpha$ intersecting $Z$.
\item $\Delta_Z$ is  a deformation retract of its subcomplex $\Delta_T$.

\end{enumerate}
\end{lemma}

\begin{proof} We use induction on the number of blow-ups $k$. If $k=0$ then both statements are obvious.
We can assume by the induction that
$\Delta_Z\subset \Delta_T$ and $\Delta_Z$ corresponds to those of intersection components of $T$ which intersect $\overline{Z}$. 

By the proof of the previous theorem each  component $D^T_\alpha$ is created  exactly after  the same blow-up   as the component $D^Z_\alpha$, with $D^Z_\alpha=D^T_\alpha\cap Z$. Morover the component $D^T_\alpha$ (or $D^Z_\alpha$)
will  vanish if the center of the blow-up $C$ contains $D^T_\alpha$ (or $D^Z_\alpha$). Thus if the component $D^T_\alpha$ vanishes then  $D^Z_\alpha$ also does. If $C$ contains $D^Z_\alpha$  but not $D^T_\alpha$ then the component $D^Z_\alpha$ will vanish while $D^T_\alpha$  will become disjoint from the strict transform of $Z$.
Let $C\subset Z$ be a smooth center having SNC with $ZS$. The new complexes $\Delta_{Z'}$, and $\Delta_{T'}$ are obtained from $\Delta_Z$ and $\Delta_T$ by the cone extensions at 
$(\Delta_{Z,C},\Delta^0_{Z,C})$ and $(\Delta_{T,C},\Delta^0_{T,C})$ respectively: 
$$\Delta_{Z'}=\Delta_Z\setminus \Delta^0_{Z,C}\cup \,v_C*(\Delta_{Z,C}\setminus \Delta^0_{Z,C}),$$ 
$$\Delta_{T'}=\Delta_T\setminus \Delta^0_{T,C}\cup \,v_C*(\Delta_{T,C}\setminus \Delta^0_{T,C})$$

Note that $$\Delta_{Z,C}=\{\tau\in \Delta_Z \mid D^Z_\tau\cap C\neq \emptyset\}=\{\tau\in \Delta_T \mid D^T_\tau\cap C\neq \emptyset\}=\Delta_{T,C}$$
 On the other hand 
 $$\Delta^0_{Z,C}=\{\tau\in \Delta_Z \mid D^Z_\tau\subset C\}\supseteq \{\tau\in \Delta_T \mid D^T_\tau\subset C\}=\Delta^0_{T,C}$$

This implies immediately that $\Delta_{Z'}$ is a subcomplex of $\Delta_{T'}$.

Moreover one can retract homotopically $$|\Delta_{T'}|=|\Delta_T\setminus \Delta^0_{T,C}\cup \,v_C*(\Delta_{T,C}\setminus \Delta^0_{T,C})|$$ to
$$|\Delta_T\setminus \Delta^0_{Z,C}\cup \,v_C*(\Delta_{T,C}\setminus \Delta^0_{Z,C})|$$
Consider a simplex $\tau$ of maximal dimension in $\Delta^0_{Z,C}\setminus \Delta^0_{T,C}$. Let $v_\tau$ be its barycenter. The face 
$\tau$ can be eliminated and $|\Delta_T|$ retracted to $|(\Delta_T)'\setminus \{\tau\}|$  by  putting $v_\tau
\mapsto (1-t)v_\tau+tv_C$, and mapping all vertices identically.
Note that any simplex $\tau'\in \Delta_T$ such that $\tau\leq \tau'$ is  necessarily in $\Delta^0_{Z,C}\subset \Delta_{T,C}$, and by maximality is in $\Delta^0_{T,C}$.  Thus   $\tau$ is a maximal face in $\Delta_{T'}$ and the retraction affects only the internal points of $\tau$ collapsing them to $v_C$. 
We transform $\Delta_{T'}$ to $\Delta_T\setminus (\Delta^0_{T,C}\setminus \{\tau\})\cup \,v_C*(\Delta_{T,C}\setminus (\Delta^0_{T,C}\setminus \{\tau\}))$
 By the repeating this peocedure one by one we collapse, or eliminate all simplices in $\Delta^0_{Z,C}\setminus \Delta^0_{T,C}$, thus transforming 
$\Delta_{T'}$ into 
$\Delta_T\setminus \Delta^0_{Z,C}\cup \,v_C*(\Delta_{T,C}\setminus \Delta^0_{Z,C})$.

Note that $\Delta^0_{Z,C}$ is star closed in $\Delta_T$. If $\tau \in \Delta^0_{Z,C}$ and $\tau\leq \tau'$ for $\tau'\in \Sigma_T$. Then $\tau' \in \Delta_{T,C}=\Delta_{Z,C}\subset \Delta_Z$.  But $\Delta^0_{Z,C}$ is star closed in $\Delta_Z$, and thus $\tau'\in \Delta^0_{Z,C}$.

This implies that $|\Delta^0_{Z,C}|$ is open in $|\Sigma_T|$.
The retraction $|\Delta_T|$ to $|\Delta_Z|$ transforms an open $|\Delta^0_{Z,C}|$ identically. Thus it extends to  the retraction of $|\Delta_T\setminus \Delta^0_{Z,C}|$ to $|\Delta_Z\setminus \Delta^0_{Z,C}|$
and the retraction of $|\Delta_T\setminus \Delta^0_{Z,C}\cup \,v_C*(\Delta_{Z,C}\setminus \Delta^0_{Z,C})|$ to
$|(\Delta_Z)'|=|\Delta_Z\setminus \Delta^0_{Z,C}\cup \,v_C*(\Delta_Z,C)\setminus \Delta^0_{Z,C})|$.

\end{proof}
\begin{corollary}\label{impo2} Consider any embedding $ Z\subset T$ of smooth varieties, and let $\overline{Z}\to Z$ be the sequence of blow-ups of centers having SNC with exceptional divisors, and $\overline{T}\to T$ be the induced sequence of blow-ups on $T$.
 Denote the exceptional (SNC) divisor of $\overline{Z}\to Z$ by $D^Z$,  and by $D^T$ be the exceptional divisor of $\overline{T}\to T$. Let  
$D^Z_1\subset\ldots\subset  D^Z_k\subset D^Z$ and  $D^T_1\subset\ldots\subset  D^T_k\subset D^T$ denote filtrations of divisors such that 
\begin{enumerate}
\item $D_i^Z=D_i^T\cap \overline{Z}$
\item  There is an inclusion of complexes $\Delta^Z_i:=\Delta(D^Z)\subset \Delta_T:=\Delta(D^T)$.
\item  Each intersection component $D^Z_\alpha$ is contained in a unique minimal component $D^T_\alpha$. Moreover the components $D^Z_\alpha$ are in bijective correspondence with the components of $D^T_\alpha$ intersecting $Z$.
\item $\Delta^Z$ is  a deformation retract of its subcomplex $\Delta^T$. Moreover the deformation transforms 
each $\Delta^T_i$ into $\Delta^T_i$

\end{enumerate} Let  $C\subset \overline{Z}$ be a center having SNC with $D^Z$ (and thus $D^T$). Denote by $D^{T'}$ and $D^{Z'}$ the full transformation of the divisors  $D^T$ and $D^Z$. Consider the natural filtration of $D^{T'}$ (or $D^{Z'}$) by 
$D^{T'}_i$ and $D^{Z'}_i$, where $D^{Z'}_i$ is the inverse image of  $D^{Z'}_i$, and $D^{Z'}_i$. 
Then the above conditions are satisfied after the  blow-up of $C$ for 

$D^{Z'}_1\subset\ldots\subset  D^{Z'}_k\subset D^{Z'}$ and  $D^{T'}_1\subset\ldots\subset  D^{T'}_k\subset D^{T'}$.

\end{corollary}
\begin{proof} The blow-up of the center $C$ defines cone extensions  $\Delta_{Z'}$ and $\Delta_{T'}$ of $\Delta^Z$ and $\Delta^T$ respectively  at $(\Delta_{Z,C},\Delta^0_{Z,C})$ and $(\Delta_{T,C},\Delta^0_{T,C})$. Moreover by the previous proof we know that $\Delta_{Z,C}=\Delta_{T,C}$ and $\Delta^0_{Z,C}$ contains $\Delta^0_{T,C}$ and is a star closed subset of both $\Delta^Z$ and $\Delta^T$.
Correspondingly $\Delta_{Z,C,i}=\Delta_{Z,C}\cap \Delta^Z_i$, and $\Delta^0_{Z,C,i}=\Delta^0_{Z,C}\cap \Delta^Z_i$.

Consider  3 cases. 

{\bf Case 1.} Assume the center  $C$ is not a component of $D_Z$.  Then it is also not a  component of  $D_T$. The complexes  
$\Delta_{Z'}$ and $\Delta_{T'}$ are obtained from $\Delta_Z$ and $\Delta_{T'}$ by pure cone extension at $\Delta_{Z,C}=\Delta_{T,C}$. Then the retraction of $|\Delta_{T}|$ to $|\Delta_{Z}|$ is identical on $|\Delta_{Z,C}|=|\Delta_{T,C}|$ and extends identically  to  the retraction of $|\Delta_{T'}|$ to $|\Delta_{Z'}|$. This is compatible with the retraction of filtration $|\Delta_{iT'}|$ to $|\Delta_{iZ'}|$.

{\bf Case 2.} Assume  the center $C$ is a component of $D^Z$ corresponding to $\sigma_C\in \Delta_Z$ but is not a component of $D_T$.  Then as before $\Delta_{Z,C}=\Delta_{T,C}$, and $\Delta^0_{Z,C}=\Star(\sigma_C,\Delta_{Z})$,  with  $\Delta^0_{T,C}=\emptyset$.

Then $|\Delta_{T'}|$ rectracts to $|\Delta_{T}|$ which further retracts by the inductive assumption to $|\Delta_{Z'}|=|\Delta^{Z}|$ as $\Delta_{Z'}$ is a star subdivision of $\Delta_{Z}$. The subcomplexes $\Delta_{T'i}$ are retracted to 
$\Delta_{T,i}$ (by moving $v$ to $v_\sigma$) and then further to
$\Delta_{Z'i}$  with $|\Delta_{Z'i}=|\Delta_{Z'i}|$. 

{\bf Case 3.} Assume the center $C$ is a component of both $D^Z$ and $D^T$ then $$\Delta_{Z,C}=\Delta_{T,C}=\Star(\sigma_C,\Delta_{Z})=\Star(\sigma_C,\Delta_{T}),$$
and  complexes $\Delta_{Z'}$ and $\Delta_{T'}$ are obtained from 
$\Delta_{Z}$ $\Delta_{T}$, by the star subdivision at $\sigma_C$. The retraction of $|\Delta_{T,C}|$ to $|\Delta_{Z,C}|$ is identical on $|\Star(\sigma_C,\Delta_{Z})|$ and defines the retarction of $|\Delta_{T',C}|$ to $|\Delta_{Z',C}|$ compatible with  filtration $(\Delta_{T'i})$.

\end{proof}
For any simplicial complex $\Sigma$ and is face $\sigma$ let $\overline{\sigma}$ be the set of all faces of ${\sigma}$. In particular the set
$\overline{\sigma}$ is subcomplex of $\Sigma$.
If $\tau\leq \sigma$ then the face map $i_{\tau,\sigma}: \tau \to \sigma$ extends to  a map (inclusion) between complexes $\overline{i_{\tau,\sigma}}: \overline{\tau} \to \overline{\sigma}$
 Moreover $\Sigma$ is a union of the subcomplexes $\overline{\sigma}$ for $\sigma\in \Sigma$, with the natural identifications of simplices determined by the maps $\overline{i_{\tau,\sigma}}$.

\begin{definition}
The {\it simplicial product} of simplices $\sigma_1=\sigma_1(e_0,\ldots,e_k)$ and $\sigma_2=\sigma_2(e'_0,\ldots,e'_n)$ with vertices respectively $e_0,\ldots,e_k$, and $e'_0,\ldots,e'_n$, is the  simplex $$\sigma_1\otimes\sigma_2=\sigma_1\otimes\sigma_2(e_{00},\ldots,e_{10},\ldots,e_{kn}),$$  
with the set of vertices $\{e_{00},\ldots,e_{10},\ldots,e_{kn}\}$ corresponding to the product $\{e_0,\ldots,e_k\}\times\{e'_0,\ldots,e'_n\}$.
\end{definition}

If $\tau_1\leq \sigma_1$, and $\tau_2\leq \sigma_2$ are defined by the inclusion of vertices then
there is a face map $$i_{\tau_1\otimes\tau_2,\sigma_1\otimes\sigma_2}:  \tau_1\otimes\tau_2 \to \sigma_1\otimes\sigma_2$$
which is also defined  by the inclusion of vertices. 
\begin{definition}
The {\it simplicial product} of simplicial complexes  $\Sigma_1$ and $\Sigma_2$ is  the simplicial complex
$\Sigma_1\otimes\Sigma_2$ which is a union of all the complexes $\overline{\sigma_1\otimes\sigma_2}$ where $\sigma_i\in \Sigma_i, i=1,2$, with inclusion maps $\overline{i_{\tau_1\otimes\tau_2,\sigma_1\otimes\sigma_2}}:  \overline{\tau_1\otimes\tau_2} \to \overline{\sigma_1\otimes\sigma_2}$ for $\tau_1\leq \sigma_1$, and $\tau_2\leq \sigma_2$
\end{definition}

The definition is, essentialy equiavalent to the definition of the product in the category of simplicial sets.

This notion is closely related to the {\it cartesian product} of 
simplicial complexes which is a polyhedral complex:
$$\Sigma_1\times\Sigma_2:=\{\sigma_1\times\sigma_2\mid \sigma_i\in \Sigma_i, i=1,2 \}$$

The face maps in $\Sigma_1\times\Sigma_2$ are given by the products of the face maps in $\Sigma_i$. As before $\Sigma_1\times\Sigma_2$ can be represented as the union of the complexes $\overline{\sigma_1\times\sigma_2}=\overline{\sigma_1}\times\overline{\sigma_2}$. 

\begin{lemma} If $D_1$ and $D_2$ are SNC divisors then $D_1\times D_2$ is a variety with SNC. Moreover $$\Delta(D_1\times D_2)=\Delta(D_1)\otimes\Delta(D_2).$$
\end{lemma}
\begin{proof} If $D_1=\sum D_{1i}$ and $D_2=\sum D_{2i}$ then
$D_1\times D_2=\bigcup D_{1j}\times D_{2i}$ which defines the dual complex $\Delta(D_1)\otimes\Delta(D_2)$.

\end{proof}

Observe that the star subdivision of a polyhedral complex $\Sigma$ at any of its vertices $v$ creates a new complex 
$$\Sigma'=\Sigma\setminus \Star(v,\Sigma)\cup v*\Link(v,\Sigma),$$ with the same set of vertices and all the faces containg $v$ are of the form $v*\sigma$, with $v$ independent of $\sigma$. Thus applying the star subdivisions to all the vertices of $\Sigma$ creates its triangulation, that is a simplicial complex $\Sigma^{simp}$  with the same geometric realization.

Let $\Sigma_1, \Sigma_2$ be two simplicial complexes. For any faces  $\sigma_1\in \Sigma_1$ and $\sigma_2\in \Sigma_2$ the bijective correspondence between the vartices defines  a natural projection $\phi_{\sigma_1,\sigma_2}: {\sigma_1\otimes\sigma_2}\to \sigma_1\times \sigma_2$ which agrees on the faces and extends to the map  $$\phi_{\sigma_1,\sigma_2}: \overline{\sigma_1\otimes\sigma_2}\to \overline{\sigma_1\times\sigma_2}$$ and 
which globalizes to $$\phi: \Sigma_1\otimes\Sigma_2\to \Sigma_1\times\Sigma_2,$$
and the map of its geometric realizations:
$$|\phi|: |\Sigma_1\otimes\Sigma_2|\to |\Sigma_1\times\Sigma_2|.$$

Note that the map $\phi$ is the product of two natural projections $\pi_1:\Sigma_1\otimes\Sigma_2\to \Sigma_1$ and $\pi_2:\Sigma_1\otimes\Sigma_2\to \Sigma_2$, induced by the projections of the vertices.
On the other hand the inclusion maps defined on vertices induces 
a unique inclusion map 
$$i_{\sigma_1,\sigma_2}:  \overline{(\sigma_1\times\sigma_2)}^{simp}\to \overline{\sigma_1\otimes\sigma_2}$$ 
extending to $$i:  (\Sigma_1\times\Sigma_2)^{simp}\to\Sigma_1\otimes\Sigma_2.$$
and the closed embedding of the geometric realizations:
$$|i|: |(\Sigma_1\times\Sigma_2)^{simp}|=|(\Sigma_1\times\Sigma_2)^{simp}|\to |\Sigma_1\otimes\Sigma_2|.$$
\begin{lemma}\label{product}
The maps $|i|$ and $|\phi|$ defines a homotopy equivalence between the geometric realizations of both complexes. That is 
$$|\Sigma_1\otimes\Sigma_2| \simeq |(\Sigma_1\times\Sigma_2)^{simp}|=|(\Sigma_1\times\Sigma_2)|$$

\end{lemma}
\begin{proof} Indeed $\phi i=id_{(\Sigma_1\times\Sigma_2)^{simp}}$. On the other hand consider the homotopy which takes $y\in |\Sigma_1\otimes\Sigma_2|$ to 
 $y_t=  (1-t)y+t \cdot i(\phi(y))$ which defines homotopical equivalence of $id_{|\Sigma_1\otimes\Sigma_2|}$ and $i\phi:|\Sigma_1\otimes\Sigma_2|\to |\Sigma_1\otimes\Sigma_2|$.

\end{proof}

\section{Combinatorial type of varieties}

The following theorem gives us a  functorial correspondence between quasiprojective varieties and the their {\it combinatorial types}, that is the homotopy types of certain associated complexes. It is  a generalization of the correspondence between SNC divisors and their dual complexes. While the notion of the combiantorial type is defined for any schemes of finite types, the  induced maps between combinatotial types are considered  only for quasiprojective varieties.

\begin{theorem} Let $K$ be a field of characteristic zero.
One can associate with any scheme  $X$ of finite type over $K$ its {\bf combinatorial type} $\Sigma(X)$ that is a unique canonical homotopy type of a finite complex such that 
\begin{enumerate}
\item If $X$ is any variety with SNC crossings then  $\Sigma(X)$ is the homotopy type of its dual complex $\Delta(X)$.

\item For any morphism $\phi: X\to Y$ of the quasiprojective schemes of the finite type there is a canonial topological 
map $\Sigma(f): |\Sigma(X)|\to |\Sigma(Y)|$ defined uniquely up to homotopy.
\item  For any  morphisms $Z \buildrel{\phi}\over \to Y\buildrel{\psi}\over \to X$ of quasiprojective varieties the maps of the topological spaces:
$$\Sigma(\psi\phi): |\Sigma(Z)|\to |\Sigma(X)|\,\,\,\mbox{and}\,\,\, \Sigma(\psi)\Sigma(\phi) : |\Sigma(Z)|\to |\Sigma(X)|$$ are homotopy equivalent 

\item  $\Sigma(X\times Y)=\Sigma(X)\otimes \Sigma(Y)\simeq |\Sigma(X)\times \Sigma(Y)|$, and for the projection $\pi_Y: X\times Y\to Y$, the induced map $\Sigma(\pi_Y): |\Sigma(X)|\times |\Sigma(Y)|\to |\Sigma(Y)|$ is homotopy equivalent to the projection \\$\pi_{|\Sigma(Y)|}: |\Sigma(X)|\times |\Sigma(Y)|\to |\Sigma(Y)|$.

\item If $\phi: Z_1\dashrightarrow Z_2$ is a proper birational map of smooth varieties which transforms closed subvariety $X_1\subset Z_1$ to $X_2\subset Z_2$ then $\Sigma(X_1)=\Sigma(X_2)$.

\item If $X=X_1\cup X_2$  is a union of the closed subvarieties $X_1$ and $X_2$ then 
$\Sigma(X)$ is a push out of $\Sigma(X_1)\leftarrow \Sigma(X_1\cap X_2)\to \Sigma(X_2)$.  
\item
 If $X$ is a complex projective variety then $W_0(H^i(X,\CC))\simeq H^i(\Sigma(X),\CC)$.

\end{enumerate}

\end{theorem}

The theorem will be proven in a few steps. First we introduce the definition of the combinatorial type for quasiprojective varieties.

{\bf Step 1.} {\bf The combinatrial type of quasiprojective  varieties.}

 Given a  quasiprojective scheme of finite type $X$  one  can associate with it its combinatorial type. Embed $X$ into a smooth quasiprojective   variety $Z$, and consider  the extended Hironaka  desingularization $\overline{Z}\to Z$ of $X\subset Z$. It defines an SNC divisor $D\subset Z$. 
We define the {\it  combinatorial type} of $X$, denoted $\Sigma(X)$ to be the homotopy type of the dual complex $\Delta(D)$.

We will  show that  the combinatorial type $\Sigma(X)$ of $X$ does not depend upon the embedding and desingularization or principalization which can be used instead of desingularization. 
 \begin{proposition} \label{22} Let $s$ be the number of blow-ups used in the extended canonical desingularization of $X$, and $n=\dim(X)$ be its dimension. 
 For any extended canonical desingularization of $X\subset Z$, where $m>s+n$ the dual complex $\Delta(D_X^Z)$ of the exceptional divisor $D_X^Z$ is completely determined by the intersections of the centers of blow-ups on the strict transform of $X$ with excepional divisors. Thus it depends only on the strict transforms and it is independent of the embedding. 
 
 More precisely denote by $D^i$ for $i=1,\ldots,s$ the exceptional divisors on the  varieties $Z^i$ obtained by consecutive blow-ups of $Z$ with $D^s=D$.
 Then $\Delta^{can}(X)=\Delta(D_X^Z)$ is obtained by the sequence of the pure cone extensions of $\Delta^i=\Delta(D^i)$ at the canonical subcomplexes 
 $\Delta^i_C$, where $$\Delta^i_C=\{\tau \in \Delta^i\mid D_\tau\cap C \neq 0\}$$ is independent of the ambient varieties $Z^i$ (and depends only on the intersections of $D^i$ with the strict transforms of $X$).
 
\end{proposition}
\begin{proof}
In fact  we can assume that the exceptional divisors $D^\cdot_1,\ldots, D^\cdot_k$ are ordered by the sequence of blow-ups. The component $D^{  \cdot s}_{j_1,\ldots,j_k}$ is created when  the component $D^{  \cdot s'}_{j_1,\ldots,j_{k-1}}$ intersects the center $C_{j_k}$ on the strict transform $X_{j_k}$ of $X$. Moreover it becomes the exceptional divisor of the component of the intersection $C_{j_k}\cap D^{s'}_{j_1,\ldots,j_{k-1}}$  on $D^{s'}_{j_1,\ldots,j_{k-1}}$. Once the coponent is created it will define a nonsingular subvariety of codimension $\leq s$ so of dimension $>n$. The consecutive centers are contained in the strict transforms of $X$, and are of dimension $\leq n$ so they won't affect the created components $D^s_{j_1,\ldots,j_k}$.
Thus the component $D^s_{j_1,\ldots,j_k}$ depends entirely on its creation on the strict transform of $X$ and is independent on the embedding. Similarly the intersection component $D^s_{j_1,\ldots,j_k}\cap \overline{X}$ depends on the strict transform $\overline{X}$ of $X$, and the restrictions of the exceptional divisors to $\overline{X}$.

\end{proof}

\begin{proposition} For any quasiprojective variety $X$ there exists a canonical dual complex $\Delta^{can}(X)$ which is functorial with respect to smooth morphisms such that 
\begin{enumerate}
\item For any embedding $X\hookrightarrow Z$ into a smooth variety $Z$ of sufficiently large dimension $\Delta^{can}(X)=\Delta(D^Z_X)$

\item Any  smooth morphism $\phi: Y\to X$  defines the map of complexes $\Delta^{can}(\phi): \Delta^{can}(Y)\to \Delta^{can}(X)$
\item For any smooth morphisms $Z\to Y \to X$ we have  that
$$\Delta^{can}(\psi\phi)=\Delta^{can}(\psi)\Delta^{can}(\phi)$$ 
 \item If $\phi: U\to X$ is an open inclusion then $\Delta^{can}(\phi) :\Delta^{can}(U)\subset\Delta^{can}(X)$ is a canonical inclusion of complexes.
\end{enumerate}

\end{proposition}
\begin{proof} Follows immediately from the canonicity of the  Hironaka (extended) desingularization. Note that any simplex $\Delta^s_{j_1,\ldots,j_k}$ in $\Delta^{can}(Y)$ can be idntified canonically with the subset $D^s_{j_1,\ldots,j_{k-1}}\cap C_k$ on the strict transform $X_k$ of $X$. This gives functoriality of $\Delta(\phi)$  defined for the smooth maps $\phi: X\to X'$, since such a map necessairily extends to $X_k\to X'_k$, defining the unique maps between components transforming divisor into divisors.

\end{proof}

By the proposition one can construct the canonical dual complex for any scheme $X$ of finite type. Consider an open affine cover $(X_i)$ of $X$  and construct $\Delta^{can}(X)$   by gluing the complexes $\Delta^{can}(X_i)$ along $\Delta^{can}(X_i\cap X_j)$.

\begin{corollary}  If $X=X_1\cup X_2$ then there exist canonical inclusion $\Delta^{can}(X_1\cap X_2)\subset \Delta^{can}(X_i)$. Then $\Delta^{can}(X)$ is obtained by the canonical glueing of $\Delta^{can}(X_i)$ along $\Delta^{can}(X_1\cap X_2)$.
\end{corollary}
\begin{theorem} If $X$ is a variety with SNC then $\Sigma(X)$ is a homotopy type of the dual complex $\Delta(X)$.
\end{theorem} 
\begin{proof} The complex $\Delta^{can}(X)$ is obtained from $\Delta(X)$ by the sequence of the canonical cone extensions of the form (4) and (5) from Proposition \ref{cone2} corresponding to the blow-ups in the extended Hironaka desingularization. These operations do not change the homotopy type.
\end{proof}

\begin{proposition}\label{impo3} Let $X\subset Z$ be any  closed embedding of $X$ as a closed subscheme of
a smooth variety $Z$. Let ${Z}_0\to Z$ be any proper birational map from smooth variety ${Z}_0$, such that the proper  transform of $X$ is  a variety $D^{Z_0}_X$ with SNC  on
${Z}_0$. Then the combinatorial type $\Sigma(X)$ is the homotopy type the dual complex $\Delta(D^{Z_0}_X)$.
\end{proposition}

\begin{proof}  Consider a further closed embedding $Z\subset T=:Z\times \bP^n$ into a closed embedding  of sufficiently large dimesion $n$.
Then the extended canonical embedded desingularization $\overline{Z}\to Z$ of $X\subset Z$ defines an extended canonical embedded desingularization $\overline{T}\to T$ of $X\subset T$. It transforms $X$ into $D^{Z}_{X}$ on $\overline{Z}$ and $D^T_{X}$ on $\overline{T}$.
Moreover the dual complex $\Delta(D^T_{X})=\Delta^{can}(X)$ on $\overline{T}$ is independent of the embedding by Proposition \ref{22}.
On the other hand, by Lemma \ref{impo}, the dual complex $\Delta(D^Z_{X})$ is a deformation retract of $\Delta(D_{X}^{T})$.
To finish the proof we need  the following theorem 
\begin{theorem}(\cite{abw}, see also \cite{step},\cite{payne})
\label{intro}
Suppose that  $X$ is a smooth  variety, and
$D\subset X$ be its closed subcsheme  which is a  variety with SNC. Then the homotopy type of the
dual complex of $D$  depends only the  the
complement $X-D$, and in fact only on its proper birational class.
\end{theorem}
By Theorem \ref{intro}, $\Delta(D^{Z_0}_{X})$ is homotopy equivalent to $\Delta(D^Z_{X})$, and thus to $\Delta^{can}(X)=\Delta(D^T_{X})$. 
\end{proof}



{\bf Step 2.} {\bf Maps of complexes associated with closed embeddings.}

Let $i: Y\subset X$ be closed embedding of quasiprojective varieties. Consider the embeddings $Y\subset X\subset Z$ into smooth variety $Z$ of a sufficiently large dimension  and the canonical extended designularization $Z_1\to Z$  $Y\subset Z$. It produces  a divisor $D^{Z_1}_Y$ on ${Z}_1$. Next apply the canonical extended desingularization $\overline{Z}\to Z_1$ of the proper transform $X_1\subset Z_1$ with SNC to $D^{Z_1}_Y$.

 This defines  a pair of SNC divisors $D^Z_Y\subset D^Z_X$ and the inclusion $\Delta(D^Z_Y)\hookrightarrow \Delta(D^Z_X)$.
 As in Theorem \ref{22} the process is canonical and produces
a unique pair of complexes $$\Delta^{can}(Y; (Y,X)):=D^Z_Y \subset \Delta^{can}(X; (Y,X)):=\Delta(D^Z_X).$$

Observe that, by Proposition \ref{impo3}, $\Sigma(Y)$ and $\Sigma(X)$ are the homotopy type of $\Delta^{can}(Y; (Y,X))$ and  $\Delta^{can}(X; (Y,X))$. We define the map $\Sigma(i):\Sigma(X)\to \Sigma(Y)$ 
of the combinatorial types associated with the embedding $i$ to be  the homotopy type of the inclusion $$\Delta^{can}(Y; (Y,X))\subset\Delta^{can}(X; (Y,X))$$


The construction can be 
extended to a sequence closed embeddings $X_1 \buildrel{i_1}\over\hookrightarrow X_2\buildrel{i_2}\over\hookrightarrow  \ldots \buildrel{i_{k-1}}\over\hookrightarrow  X_k\subset Z$  where $Z$ is an ambient smooth field of sufficiently large dimension. By
applying   
 the  canonical embedded desingularization  of $X_1$ followed 
by the canonical embedded desingularization of $X_2$ and etc, up $X_k$ 
we produce the sequence of SNC divisors $D^Z_{X_1}\subset \ldots\subset D^Z_{X_k}$. As before ( Theorem \ref{22}) the process is canonical and produces and produces
a unique sequence  of complexes $$\Delta^{can}(X_1; (X_1,\ldots,X_k)):=\Delta(D^Z_{X_1}) \subset \ldots\subset \Delta^{can}(X_k; (X_1,\ldots,X_k)):=\Delta(D^Z_{X_k})\eqno (*)$$

Note that the by Proposition \ref{impo3}, the homotopy type of $\Delta^{can}(X_k; (X_1,\ldots,X_k))$ is equal to $\Sigma(X_i)$. Recall the following result (A simple consequence of the Weak factorization Theorem):
\begin{lemma} (\cite{payne},\cite{abw})\label{ppp} Let $X_1\subset \ldots\subset X_k\subset Z$  be embeddings of  $X_i$ into smooth $Z$, then for any proper birational transformation $Z'$ of $Z$ into SNC divisors (or varieties with SNC) $X'_1\subset \ldots\subset X'_k\subset Z'$ on the smooth $Z'$,
the induced embeddings of the dual complexes $\Delta(X'_1)\subset \ldots\subset \Delta(X'_k)$ are homotopy equivalent.
\end{lemma}

By the the lemma, the sequence (*) defines the sequence of the induced maps of the cominatorial types
$$\Sigma(X_1)\buildrel{\Sigma(i_1)}\over\hookrightarrow\Sigma(X_2)\buildrel{\Sigma(i_2)}\over\hookrightarrow \ldots \buildrel{\Sigma(i_{k-1})}\over \hookrightarrow \Sigma(X_k)\eqno(**)$$

This observation can be further extended
\begin{lemma}\label{impo5}
Let $X_1 \buildrel{i_1}\over\hookrightarrow X_2\buildrel{i_2}\over\hookrightarrow  \ldots \buildrel{i_{k-1}}\over\hookrightarrow  X_k\subset Z$  be embeddings of  $X_i$ into smooth $Z$, then for any proper birational transformation of $Z$ taking $X_i$ into SNC divisors (or varieties with SNC) $X'_1\subset \ldots\subset X'_k\subset Z'$ on the smooth $Z'$,
the induced embeddings of the dual complexes $$\Delta(X'_1)\subset \ldots\subset \Delta(X'_k)\eqno{(***)}$$ define the induced sequence of combinatorial types $$\Sigma(X_1)\buildrel{\Sigma(i_1)}\over\hookrightarrow\Sigma(X_2)\buildrel{\Sigma(i_2)}\over\hookrightarrow \ldots \buildrel{\Sigma(i_{k-1})}\over \hookrightarrow \Sigma(X_k)$$
\end{lemma}
\begin{proof} The embeddding of $X_i$ into $Z$ can be further  
extended by embedding $Z\to T\times \bP^n$ into a smooth ambient variety $T$ of sufficiently large dimension. Applying 
the extended Hironaka desingularization to $X_i$ on $T$ gives 
a sequence of the dual complexes $$\Delta^{can}(X_1; (X_1,\ldots,X_k)) \subset \ldots\subset \Delta^{can}(X_k; (X_1,\ldots,X_k))$$
which is homotopy equivalent to the sequence (**). Now by Corollary \ref{impo2} the sequnece is homotopy equivalent to the sequence of the embedding of the induced dual complexes on the subvariety $Z$. The latter, in view of Lemma \ref{ppp}, is homotopy equivalent to the sequence (***).
\end{proof}



\bigskip 
{\bf Step 3}. {\bf Maps associated with projections}

Let $X$ and $Y$ be quasiprojective varieties and $X\subset T$ and $Y\subset V$ be  closed embeddings into a smooth ambient varieties of sufficiently large dimension.  
Consider the extended Hironaka desingularization of  $X\subset T$ and $Y\subset V$ producing SNC divisors $D_X$, and $D_Y$. This defines  a variety with SNC $D_X\times D_Y$ corresponding to  the dual complex $\Delta(D_X)\otimes \Delta(D_Y)$ whose homotopy type is equivalent to $|\Delta(D_X)|\times|\Delta(D_Y)|$ and thus of $\Sigma(X)\times \Sigma(Y)$. In the case of arbitrary schemes of finite types
we consider affine covers $(X_i)$ of $X$, and $Y_i$ of $Y$.
Then $\Delta^{can}(X\times Y)$ is obtained by the canonical glueing of $\Delta^{can}(X_i\times Y_i):=\Delta^{can}(X_i) \otimes \Delta^{can}(Y_i)$. This allows to construct
$$\Delta^{can}(X\times Y)=\Delta^{can}(X_i) \otimes \Delta^{can}(Y_i),$$
 and, in view of  Lemma \ref{product},  leads to a homotopy equivalence: 
  $$\Sigma(X\times Y)\simeq \Sigma(X)\times \Sigma(Y).$$
In the particular case, when $Y$ is smooth we get the isomorphism $\Sigma(X\times Y)=\Sigma(X)\times \Sigma(Y)\simeq \Sigma(X)$. We shall describe this isomorphism in a functorial way.
 
\begin{lemma}\label{map} Let $D$ be an SNC divisor on a smooth variety $Z$ and  let $E$ be its closed subscheme which is avriety with  SNC, and let $i:E\subset D$ be the closed embedding with the induced map of combinatorial types $\Sigma(i): \Sigma(E)\to \Sigma(D)$ (introduced  in the previous Step). 
 
 Assume that each maximal irreducible components $E_i$ of $E$ is contained in a unique maximal component $D_i$ of $D$. Then any component $E_\alpha^{s}$ is contained in the relevant component $D_\alpha^{s'}$ for a unique $s'$. This induces a map on vertices of the dual complex $\Delta(E)$ which  extends uniquely to the map 
$ \Delta(E)\to  \Delta(D)$ and its geometric realization    $\alpha: |\Delta(E)|\to  |\Delta(D)|$ which represents the map between combinatotial types of $\Sigma(E)$, and $\Sigma(D)$. Then  the map $\alpha$ is homotopy equivalent to $\Sigma(i)$.

\end{lemma}
\begin{proof} By definition one  can identify  $\alpha(\Sigma(E))$  with a subcomplex of $\Sigma(D)$.
Consider the blow-up of the maximal component $E_i$ which is not divisorial.  It defines a cone extension  $\Sigma(E')$ of $\Sigma(E)$ at $(\Delta_{E_i,\Sigma(E)},\Delta^0_{E_i,\Sigma(E)})$
with $\Delta_{E_i,\Sigma(E)}=\overline{\Star(v_i,\Sigma(E))}$
$\Delta^0_{E_i,\Sigma(E)}\supseteq {\Star(v_i,\Sigma(E))}$
(see Lemma \ref{cone2}).

It replaces  $v_i$ and all the simplices in $\Delta^0_{E_i,\Sigma(E)}=\Star(v_i,\Sigma(E))$ with a  new vertex $v$ and the  new simplices in $v*(\Delta_{E_i,\Sigma(E)}\setminus \Delta^0_{E_i,\Sigma(E)})$. We obtain the complex $\Sigma(E')$ which is can be identified with the subcomplex of $\Sigma(E)$ by putting $v\to v_i$, and which is homotopy equivalent to the latter one (see Lemma \ref{cone2}).

On the other hand, the center $E_i$ defines the set
$\Delta_{E_i,\Sigma(D)}$ which is contained in $\overline{\Star(\alpha(v_i),\Sigma(D))}$. Moreover, by the definition, $\alpha(\Delta_{E_i,\Sigma(E)})\subset \Delta_{E_i,\Sigma(D)}$.  
 The new complex $\Sigma(D')$ is obtained from $\Sigma(D)$ by the cone extension at $(\Delta_{C,\Sigma(D)},\emptyset)$. The vertex $v$ corresponds to the exceptional divisor $\overline{E}_i$. The inclusion $E'\subset D'$ 
 defines a new map $\alpha': |\Sigma(E')|\to  |\Sigma(D')|$ can be described by glueing two maps
 $\alpha_0': \Sigma(E)\setminus \Delta^0_{E_i,\Sigma(E)}\to \Sigma(D)$
  and $\alpha'_1: v*(\Delta_{E_i,\Sigma(E)}\setminus \Delta^0_{E_i,\Sigma(E)})\to v*\Delta_{C,\Sigma(D)}$.
 
 The map is homotopy $\alpha'$ equivalent to  $\alpha$, as we can move $v$ to $\alpha(v_i)$. By the blowing of all the nondivisorial maximal component we create a closed embedding of SNC divisors $\overline{E}\subset \overline{D}$. Then the induced map $|\Sigma(\overline{E})|\to  |\Sigma(\overline{D})|$ is homotopy equivalent to $\alpha$ and $\Sigma(i)$.

\end{proof}

\begin{lemma}\label{map2}
Let $Z$ be a smooth variety with a fixed point $z\in Z$. Consider the  closed embedding   $j_z: X\simeq X\times \{z\}\hookrightarrow X\times Z$. Then the induced map $\Sigma(j_z): \Sigma(X)\to  \Sigma(X\times Z)$ is an isomorphism. 
\end{lemma}
\begin{proof}
Let $X\subset T_X$ be a closed embedding into smooth variety.
 Consider the canonical principalization $\overline{T}_X\to T_X$ of $X\subset T_X$. It induces  the closed embedding $D_X\hookrightarrow  D_X\times Z$ of varieties with SNC. That is, by Lemma \ref{map}, $\Sigma(D_X)\to  \Sigma(D_X\times Z)$ is an isomorphism.

\end{proof}

\begin{lemma}\label{impo6}
Let $i: X\subset Y$ be a  closed embedding of quasiprojective varieties and $Z$ be a  smooth variety with a fixed point $z\in Z$  consider the induced morphism  $i_Z: X\times Z\subset Y\times Z$ extended by the identity on $Z$ and closed embeddings $j_{X,z}: X\simeq X\times\{z\}\subset X\times Z$, and $j_{Y,z}: Y\simeq Y\times\{z\}\subset Y\times Z$. Then the commutative diagram 

$$\begin{array}{cccl}

X&\hookrightarrow & Y&  \\
\dar  && \dar &\\
 X\times Z& \hookrightarrow &Y\times Z &\end{array}$$
induces the commuative diagram of the maps of complexes

$$\begin{array}{cccl}

\Sigma(X)&\to & \Sigma(Y)&  \\
\dar  && \dar &\\
 \Sigma(X\times Z)& \to &\Sigma(Y\times Z) &\end{array}$$

with induced vertival maps given by isomorphism.

\end{lemma}
\begin{proof}
Let $X\subset Y\subset Z_X$ be a closed embedding into smooth variety $Z_X$.
 Consider the canonical principalization $\overline{Z_X}\to Z_X$ of $X\subset Y\subset Z_X$. It induces  the closed embedding $D_X\subset D_Y\subset \overline{Z_X}$ and   $D_X\times Z\subset  D_Y\times Z \subset \overline{Z_X}\times Z$, and commutative diagram of SNC divisors
$$\begin{array}{cccl}

D_X&\hookrightarrow & D_Y&  \\
\dar  && \dar &\\
 D_X\times Z& \hookrightarrow &D_Y\times Z &\end{array}$$
and the corresponding  diagram of the dual complexes:
$$\begin{array}{cccl}

\Delta(D_X)&\hookrightarrow & \Delta(D_Y)&  \\
\dar j_{X,z}  && \dar j_{Y,z}&\\
 \Delta(D_X\times Z)& \hookrightarrow &\Delta(D_Y\times Z) &\end{array}$$ 
 
 with vertical maps given by isomorphisms, as in  Lemma \ref{map}. 
 
 
 \end{proof}

\begin{lemma} \label{impo7} Let $X, Y$ be quasiprojective varieties. Let $X\subset Z_X$ be a closed embedding into a smooth variety $Z_X$. 
Consider the map $X\times Y\to Z_X\times Y$. Let $j: Y\subset Z_X\times Y$ denote an embedding.
Then the composition of the maps  $$\Sigma(X)\times \Sigma(Y)\simeq \Sigma(X\times Y)\to \Sigma(Z_X\times Y)\buildrel{({\Sigma(j)})^{-1}}\over\to \Sigma(Y)$$  defines a projection $\pi_{\Sigma(Y)}$.
\end{lemma}
\begin{proof} Consider the closed embedding $X\times Y\subset Z_X\times Z_Y$, and the induced $D_X\times D_Y\subset \overline{Z}_X\times D_Y\subset \overline{Z}_X$. It satisfies the condition from Lemma \ref{map}.
The induced map $\alpha$  is nothin but the projection on $\Sigma(Z_X\times Y)\simeq \Sigma(Y)$. Thus, by Lemmas \ref{map}, and \ref{map2},
the embeddding 
$j: X\times Y\to Z_X\times Y$ induces the map $\Sigma(j): \Sigma(X\times Y)\to \Sigma(Z_X\times Y)$ which is homotopy equivalent to the projection via the identication $\Sigma(Z_X\times Y)\to \Sigma(Y)$.
\end{proof}

{\bf Step 4.} {\bf Maps associated with general morphisms}

\begin{definition} Let $f: X\to Y$  be a morphism of quasiprojetive varieties. Then the induced map $\Sigma(f): \Sigma(X)\to \Sigma(Y)$ is defined as follows. Let $X\subset Z_X$ be a closed embedding into a smooth ambient variety $Z_X$, with a fixed point $x\in X\subset Z_X$.
Consider the composition $X\buildrel{f_1}\over\hookrightarrow X\times Y\buildrel{f_2}\over\hookrightarrow Z_X\times Y$ of the graph inclusion $f_1$ followed by the induced inclusion $f_2$. The composition morphism  $X\to Z_X\times Y$ defines 
the map $$\Sigma(f): \Sigma(X)\buildrel{\Sigma(f_2f_1)}\over\to \Sigma(Z_X\times Y)\buildrel{({\Sigma(j)})^{-1}}\over\simeq \Sigma(Y)$$

\end{definition}

Moreover the map $\Sigma(f)$ is a composition of two uniquely defined maps $$\Sigma(f_1): \Sigma(X)\to \Sigma(X\times Y)\simeq \Sigma(X)\times \Sigma(Y) $$ and the projection (see Lemma \ref{impo7}) $$\pi_{\Sigma(Y)}=({\Sigma(j)})^{-1}\Sigma(f_2): \Sigma(X\times Y)\simeq \Sigma(X)\times \Sigma(Y) \to  \Sigma(Y)$$ and thus it is uniquely defined up to homotopy type.


\begin{lemma} Let $f:X\to Y$ and $g: Y\to Z$ be morphisms of quasiprojective algebraic varieties  and fix $x\in X$, $y=f(x)\in Y$ and $z=gf(x)\in Z$. Then $\Sigma(g)\Sigma(f): \Sigma(X)\to \Sigma(Z) $ and $\Sigma(gf):\Sigma(X)\to \Sigma(Z)$ are homotopy equivalent. 
 
\end{lemma}
\begin{proof}
Consider the composition of closed embeddings maps 

$$\begin{array}{cccccccl}

X\hookrightarrow &X\times Y\hookrightarrow &T_X\times Y&\hookrightarrow &T_X\times Y\times Z& \hookrightarrow & T_X\times T_Y\times Z& \\
&&\uar  && \uar &&  \uar  &\\
 && Y& \hookrightarrow &Y\times Z &\hookrightarrow &   T_Y\times Z&\\ 
&& &&  &&  \uar  &\\

&& && &&   Z&
\end{array}$$
induces, by Lemmas \ref{impo5}, \ref{impo6}, the commutative diagram of combinatorial types with vertical arrows defined by isomorphisms.

$$\begin{array}{cccccccl}

\Sigma(X)\hookrightarrow &\Sigma(X\times Y)\hookrightarrow &\Sigma(T_X\times Y)&\hookrightarrow &\Sigma(T_X\times Y\times Z)& \hookrightarrow & \Sigma(T_X\times T_Y\times Z)& \\
\uar id_{\Sigma(X)}&&\uar  && \uar &&  \uar  &\\
 \Sigma(X)&\buildrel{\Sigma(f)}\over\longrightarrow & \Sigma(Y)& \hookrightarrow &\Sigma(Y\times Z) &\hookrightarrow &   \Sigma(T_Y\times Z)&\\ 
\uar id_{\Sigma(X)}&&\uar id_{\Sigma(Y)} &&  &&  \uar  &\\

 \Sigma(X)&\buildrel{\Sigma(f)}\over\longrightarrow & \Sigma(Y)&&\buildrel{\Sigma(g)}\over\longrightarrow &&   \Sigma(Z)&
\end{array}$$
It follows from commutativity of the diagram that $\Sigma(g)\Sigma(f)$ and $\Sigma(gf)$ are homotopy equivalent. 
\end{proof}

{\bf Step 5.} {\bf Weight Deligne filtration and the combinatorial type of complex projective schemes of finite type}.

Recall that for   a (projective) simple
 normal crossing divisor $D$ , or in general a SNC variety  
  the weight zero  part exactly coincides with the cohomology  of the
  dual complex $\Delta(D)$:

 \begin{lemma}[Deligne]  $W_0(H^i(D,\CC))=H^i(\Delta(D), \CC)$
\end{lemma}

 Thus  $W_0$ 
can be understood   as the ``combinatorial part'' of the cohomology, and  as a generalization of  the cohomology  of the dual complexes of SNC varieties. Both notions can be compared via the relevant Hironaka resolutions. 
In fact we have
\begin{proposition}\label{6} Let $\phi: X\dashrightarrow X'$   be a proper birational map between complete nonsingular smooth varieties, and suppose $Z\subset X$ and $Z'\subset X'$ are two closed subvarieies such  that  $\phi(Z)=Z$. Then 
$$\dim W_0H^i(Z)=\dim W_0H^i(Z')$$

\end{proposition}

 \begin{proof} By the Weak Factorization theorem (\cite{wlodarczyk}, \cite{akmw}) the varieties  $X$ and $X'$ can be connected
by a sequence of blow-ups with smooth centers. It suffices prove the proposition in the case when $\phi: X\to X'$ 
is the blow up along a smooth center $C$.
Let $U= X\setminus Z$, $U'=X'\setminus
Z'$.  Consider the diagram
 
 $$
\xymatrix{
\ldots W_0H^{i-1}(Z)\ar[r]\ar[d] & W_0H_c^i(U)\ar[r]\ar[d]^{\simeq} & W_0H^i(X)\ar[r]\ar[d]^{\simeq} & W_0H^i(Z))\ar[d]\ldots \\ 
 W_0H^{i-1}(Z')\ar[r] & W_0H_c^i(U')\ar[r] & W_0H_c^i(X')\ar[r] & W_0H^i(Z')
}
$$
It folows from lemma 5.4 \cite{abw} that  $W_0H_c^i(U)\to W_0H_c^i(U')$ and  $W_0H_c^i(X)\to W_0H_c^i(X')$ are isomorphisms.  By the diagram and 5 lemma we get that $$W_0H^i(Z)\to W_0H^i(Z')$$  is an isomorphism.

\end{proof}

 \begin{remark} In case $Z$ and $Z'$ are SNC varieties we get a somewhat weaker  version (over $Q$) of Theorem 7.9 (\cite{abw}). On the other hand $Z$ and $Z'$ are arbitrary  dual complexes generalize the observation for to a more general situation. 
\end{remark} 

The following theorem extends the Deligne Lemma
\begin{theorem} Let $X$ be a projective  complex variety, and $\Sigma(X)$ be its combinatorial type. Then
$$H^i(\Sigma(X),\CC)\simeq W_0(H_i(X,\CC))$$ 
\end{theorem}

\begin{proof} Consider any embedding $X\to \bP^n$, and let $\overline{\sigma}: \overline{Z}\to \bP^n$ be the canonical Hironaka principalization $\overline{\sigma}: \overline{Z}\to \bP^n\times \bP^m$ of  the ideal of $X\to \bP^n$ and  $D$ be  the exceptional divisor. 
Then, by Proposition \ref{6} and Deligne Lemma, we have $$\dim(W_0(H_i(X,\CC))=\dim(W_0(H_i(D,\CC))=\dim(H_i(\Delta(D),\CC))=\dim(H_i(\Sigma(X),\CC))$$

\end{proof}



\section{Generic SNC}

Our next goal is to study  SNC morphisms which can be understood as  an extensions of the notion of smooth morphisms.

\begin{definition} Let $Y$ be an algebraic scheme over  a ground field $K$. An algebraic scheme $D$  over $Y$
together with the morphism $\psi: D\to Y$ will be  called a  {\it SNC-variety over $Y$ of dimension n} if for any point $p\in D$ there exists an open neighborhood $U\subset D$ and the morphism  $\phi: U\to D_A$, where $D_A\subset \bA^{n+1}$ is a SNC divisor in $\bA^{n+1}$ defined by the equation $x_1\cdot\ldots\cdot x_k=0$, where $k\leq n$, and such that the induced morphism
$$\phi\times \psi:  U\to D_A\times Y$$ is \'etale.
Alternatively we say that the morphism $D\to Y$ is SNC of dimension n.
An algebraic  scheme $D$ is an SNC variety of dimension $n$ if it is an SNC variety of dimension $n$ over $Spec{K}$.

\end{definition}
(In other words  the morphism $D\to Y$ is \'etale equivalent to the projection of $D_A\times Y\to Y$ along a SNC variety $D_A$.)


 If a variety  $D$ is SNC over $\Spec(K)$  then for any point $p\in D$ there exists local parameters on $D$  that is the set of functions $u_1,\ldots u_k,u_{k+1},\ldots u_n$ defining the \'etale 
morphism to $D_A\times \bA^{n-k}$ as above. We shall call the function $u_1\ldots u_k$ the local set of {\it fixed parameters}. Then, it follows from the definition that the completion of the local ring $$\widehat{\cO_{D,p}}\simeq K_p[[u_1,\ldots u_k,u_{k+1},\ldots, u_n]]/(u_1\cdot\ldots\cdot u_k),$$
where $K_p$ is the residue field of $D$ at $p$.

\begin{remark} The notion of SNC morphism is a natural extension of the smoothness of the morphism.
\end{remark}

We need the following extension of generic smoothness theorem \cite{Hartshorne}.

\begin{proposition} {\bf Generic SNC theorem} Let $D=\sum D_i$ be a  SNC variety  of dimension $n$. Consider any morphism 
$\psi: D\to Y$. There exists an open subset $V\subset Y$ such that $\psi^{-1}(V)\to V$ is SNC.
\end{proposition}
\begin{proof}
We can assume that $Y$ is nonsingular and  irreducible and $\psi$ is dominating. Consider all the nonempty intersections $D_{i_1}\cap\ldots\cap D_{i_k}$. They are nonsingular and thus by generic smoothness we can find a nonempty open affine  $V\subset Y$  such that all the induced morphisms from intersections $$\psi_{i_1\ldots i_k}: D_{i_1}\cap\ldots\cap D_{i_k}\cap \psi^{-1}(V)\to V
$$
are smooth of dimension $n+1-k$. We shall prove that   $\psi^{-1}(V)\to V$ is SNC at any point $p\in \psi^{-1}(V)$.
We can assume that $p\in \psi^{-1}(V)\cap D_{1}\cap\ldots\cap D_{k}$, and that  it does not lie  in the intersection component of the smaller dimension. Consider the local fixed parameters $u_1\ldots u_k$  at $p$ describing the components  $ D_{i}\cap\psi^{-1}(V)$ through $p$. Let $u_{k+1}\ldots u_r $ be the local parameters on $V$ such that for the smooth morphism $$\psi_{1\ldots k}: D_{1}\cap\ldots\cap D_{k}\cap \psi^{-1}(V)\to V$$ the induced morphism 
$$\phi_0:=\psi_{1\ldots k}\times (u_{k+1}\ldots u_r): D_{1}\cap\ldots\cap D_{k}\cap \psi^{-1}(V)\to  V \times \bA^r$$ is \'etale
. Denote by  $v_{r+1},\ldots,v_n $  the local parameters at $y:=\psi_{1\ldots k}(p)\in Y$, and by $$u_{r+1}=\psi^*(v_{r+1}),\ldots, u_n:=\psi^*(v_n) $$ the induced functions on $\psi^{-1}(V)$.

Then  $$ u_1\ldots u_k, u_{k+1}\ldots u_r,u_{r+1},\ldots, u_n $$ are local parameters at $p\in V$.

Denote by $v_1,\ldots,v_r$ the natural coordinates on the closed subschemes $D_A\times \bA^{r-k}\subset \bA^{r}$ defined by $v_1\cdot\ldots\cdot v_k=0$.
Consider the   morphism $$\phi:=\psi \times  (u_1\ldots u_k, u_{k+1}\ldots u_r):\quad  \psi^{-1}(V)\,\, \to Z:=V\,\,\times \,\,D_A \,\,\times \,\,\bA^{r}\,\,.$$
The induced morphism $\widehat{\psi}^*_p:\widehat{\cO_{Z,y}} \to \widehat{\cO_{D,p}}$ sends $v_i$ to $u_i$. 
Moreover, by the construction $$\widehat{\cO_{Z,y}}\simeq K_y[[v_1,\ldots,v_n]]/(v_1\cdot\ldots \cdot v_k),$$
where $K_y$ is the residue field of $Z$ at $y=\phi(p)$.
Similarly   $$\widehat{\cO_{D,p}}\simeq K_p[[u_1,\ldots,u_n]]/(u_1\cdot\ldots \cdot u_k).$$ Moreover since  $\phi_1$ is a restriction of $\phi$ the residue field $K_p$ is a finite (separable extension ) of $K_y$.
Thus  $\widehat{\cO_{Z,y}}\otimes_{K_y}K_p\simeq  K_p[[v_1,\ldots,v_n]]/(v_1\cdot\ldots \cdot v_k)$, and $\widehat{\psi}^*_p$ induces the isomorphism the complete local rings :
$$\widehat{\cO_{Z,y}}\otimes_{K_y}K_p \to \widehat{\cO_{D,p}},$$
that is 
 $$K_p[[v_1,\ldots,v_n]]/(v_1\cdot\ldots \cdot v_k)\to K_p[[u_1,\ldots,u_n]]/(u_1\cdot\ldots \cdot u_k),\quad  v_i\mapsto u_i.$$

Then $\phi: \psi^{-1}(V)\to V\times D^A\times \bA^{r}$ is \'etale and $\psi^{-1}(V)\to V$ is SNC in a neighborhood of $p$.

\end{proof}

\section{\'Etale trivializations of strictly SNC morphisms}
Recall a well known fact of trivialization of \'etale proper morphisms:

\begin{lemma} Let $\phi: X\to Y$ be a finite \'etale  morphism of smooth varieties. Then there exists a finite \'etale morphism $\overline{Y}\to Y$ of smooth varieties such that the fiber product $\overline{X}:=\overline{Y}\times_YX$ is  a finite disjoint union of the varieties $\overline{Y}_i$ isomorphic to $\overline{Y}$, and the  
the induced morphism  $$\overline{X}=\bigcup \overline{Y}_i\to \overline{Y}$$  is  an isomorphism on each component. 
\end{lemma}

\begin{proof} Let $d$ be the degree of finite morphism. We shall give a proof by induction on $d$. Consider first the extension  $Y_1:=X$  and  induced morphism $\pi_1: X_1=Y_1\times_YX\to Y_1$. There is  natural finite morphism $i: X\to X_1=X\times_YX$, such that $\pi\circ i=id_X$. Thus $i(X)$ is a closed subvariety of $X_1$ of the same dimension, which means that $i(X)$ is an irreducible and thus connected component of a smooth subscheme. In other words $X_1$ is a disjoint union $X_1=i(Y_1)\cup X'_1$, where degree $\pi'_1: X'_1\to Y_1$ is $\leq d-1$. By the inductive assumption one can find the desired extension $\overline{Y}\to Y_1$ by successive repetition of this construction which then trivializes $\pi'_1: X'_1\to Y_1$ and consequently $\pi_1: X_1\to Y_1$.

\end{proof}
\begin{definition}
A mophism $f:X\to Y$ will be called {\it strictly smooth} if it is proper and 
can be factored as $f:X\to X'\to Y$ where $g:X\to X'$ is smooth with connected fibers and $X'\to Y$ is \'etale.
A morphism $D\to Y$ will be called a {\it strictly SNC} if it is SNC and each induced morphism $f: D_\alpha\to Y$ is strictly smooth.
\end{definition}

\begin{corollary}\label{34} \begin{enumerate} 
\item Let $\phi: X\to Y$ be a  proper  morphism of smooth varieties. Then there exists a (canonical) nonempty open subset $U\subset Y$ such that $\phi^{-1}(U)\to U$ is strictly smooth.

\item Let $D$ be an SNC variety over ground field $K$ and $f: D\to Y$ be a proper 
surjective morphism. Then there exists a (canonical) open subset $U\subset Y$ such that $f^{-1}(U)\to U$ is strictly SNC.
\end{enumerate}
\end{corollary}
\begin{proof} 
(1) Consider Stein factorization  $ X\to Z\to Y$ of $\phi$, where $\psi_1: Z\to Y$ is finite  and $\psi_2: X\to Z$ has connected fibers. By  generic smoothness there exits $U\subset Y$ such that $\psi_1^(U)\to U$ is \'etale and $\psi_2^{-1}\psi_1^(U)\to \psi_1^(U)$ is smooth.
(2) Follows from (1) and from generic SNC.
\end{proof}

\begin{corollary} Let $X\to Y$ be strictly smooth morphism.
 Then there exists an  \'etale  extension $\overline{Y}\to Y$ such that $\overline{X}:=X\times_Y\overline{Y}\to \overline{Y}$ splits into a disjoint union of the irreducible components $\overline{X}=\bigcup \overline{X}_i$, where each morphism $\phi: \overline{X}_i\to \overline{Y} $ is  smooth with connected fibers.
\end{corollary}
\begin{proof}

Let $ X\to Z\to Y$ be a factorization into a smooth morphism with connected fibers and \'etale morphism.

Consider the \'etale extension $\overline{Y}\to Y$ defined for $Z\to Y$ and inducing  a trivial cover $\overline{Z}:=Z\times_Y\overline{Y}\to \overline{Y}$. This implies that
 $\overline{Z}$ is a union of irreducible components isomorphic to $\overline{Y}$.
Thus the induced morphism $\overline{X}:=X\times_Y\overline{Y}\to \overline{Z}$ is smooth and has connected fibers. Consequently $\overline{X}$ is a union of disjoint components which are smooth over $\overline{Y}$.
\end{proof}

\begin{corollary} \label{33} Let   $f: D=\sum D_i\to Y$ be  a  strictly SNC  morphism  of dimension $n$ to a smooth variety $Y$. Then there exists  a finite \'etale morphism $\overline{Y}\to Y$ such that  

\begin{enumerate}
\item The induced morphism $\overline{f}: \overline{D}:=D\times_Y\overline{Y}\to \overline{Y}$ is SNC with $\overline{D}=\bigcup \overline{D}_{j}$ being a union of irreducible components.
\item For any irreducible component	of $\overline{D}_{\alpha}=D^s_{i_1,\ldots,i_k}$, where $\alpha:=\{i_1,\ldots,i_k;s\}$ of the intersection  of $\overline{D}_{i}$ the  restriction $\overline{f}_{|\overline{D}_{\alpha}}: \overline{D}_\alpha\to \overline{Y}$ is smooth proper with irreducible fibers $\overline{D}_{\alpha x}:=\overline{f}_{|\overline{D}_{\alpha}}^{-1}(x)$.
\item Any fiber $\overline{D}_x:=\overline{f}^{-1}(x)$  of $x\in \overline{Y}$ under $\overline{f}$ is a SNC variety with maximal components $\overline{D}_{j x}$  which are intersections of the maximal components  $\overline{D}_{j}$ and $\overline{f}^{-1}(x)$.
\item There exists a bijective correspondence between the intersection components  $\overline{D}_\alpha$ of $\overline{D}$ and the intersection components of  $\overline{D}_{\alpha x}=\overline{D}_\alpha\cap \overline{f}^{-1}(x)$. In particular $\Delta({\overline{D}})=\Delta(\overline{D}_x)$.
\item  For any $x, y\in Y$, $\Delta(D_x)=\Delta(D_y)$.
\end{enumerate}

\end{corollary}

\begin{proof} 
Consider  a strictly smooth morphism from an intersection component $f_\alpha: D_\alpha\to Y$ induced by $f$ 
and its  factorization
$D_\alpha
\buildrel{g_\alpha}\over\to \Z_\alpha\buildrel{h_\alpha}\over\to Y$ into a smooth $g_\alpha$ and  \'etale morphism $Z_\alpha\to Y$.
Let $Z\to Y$ be the component which dominates  $Y$ in the product $\prod_U \widetilde{Z}_\alpha$ over $U$. Then the morphism  $Z\to Y$ is finite and \'etale.
Consider the extension $\overline{Y}\to {Y}$ for the morphism $Z\to U$ such that $\overline{Z}:=Z\times_Y\overline{Y}\to \overline{Y}$ is a trivial cover.

The  factorization $Z\to {Z}_\alpha\to Y$ determines the factorization of the trivial cover $\overline{Z}\to \overline{Z_\alpha}:=\widetilde{Z}_\alpha\times_Y\overline{Y} \to \overline{Y}$.
Thus all $\overline{Z_\alpha}=\bigcup \overline{Z}_{\alpha i}\to \overline{Y}$ are trivial covers over $\overline{Y}$. Since  the morphism $\overline{D}_\alpha:=D_\alpha\times_Y\overline{Y}
\to \overline{Z}_\alpha$ is smooth with connected fibers, the variety $\overline{D}_\alpha$ splits into components  $D_{\alpha i}$ which are  smooth over $\overline{Y}_{\alpha i}$ with connected fibers.  In other words all the connected components of $\overline{D}_\alpha$ are smooth and have connected,and  in fact,  irreducible fibers over $\overline{Y}$.

Each fiber $f^{-1}(x)$ is an SNC variety. It is a union of varieties which are intersections of the fiber and irreducible divisor  $D_{ix}=D_i\cap f^{-1}(x)$. The components intersections of these $D_{\alpha x}$ are exactly
the intersections of components of $D_{\alpha}$ with $f^{-1}(x)$. This shows that there is a bijective correspondence between the $D_{\alpha}$ and $D_{\alpha x}$ as well as their intersections proving that the dual complexes of both are the same.
In particular, the dual complexes of the fibers are the same.

\end{proof}

\section{Desingularization of morphisms and equisingular Hironaka desingularization}
\begin{theorem} {\bf Desingularization of the fibers of morphism}
 Let $\pi: X\to Y$ be any morphism over a field $K$, of characteristic zero.
There exists a canonical resolution, that is a nonsingular variety  $\overline{X}$ and a projective morphism $\phi: \overline{X}\to X$ such that the composition morphism $\overline{X}\to Y$ has all SNC  fibers. Moreover

\begin{enumerate}
\item Let $U\subset Y$ be a maximal open nonsingular subset such that $V:=\pi^{-1}(U)\to U$ is smooth.
Then
$\phi$ is an isomorphism over $V$. 
\item The exceptional locus of $\phi$ is a  SNC divisor $D$.
\item The fibers of $\overline{\psi}$ are SNC varieties and  they have SNC with $D$.
\item There exist a nonsingular locally closed (smooth) stratification $S$ of $Y$ with  the generic stratum  $U$ and such that for any nongeneric stratum $s\in S$ the set $\phi^{-1}(s)$ is a SNC divisor contained in $D$ and the morphism $\phi^{-1}(s)\to s$ is SNC.
\item If $\phi$ is proper then  the combinatorial type of the fibers $\Delta({f^{-1}(x)})$ is the same for all points in the stratum $x\in s$.
\item For each stratum there is an \'etale morphism $\overline{s}\to s$, such that the combinatorial type $\Delta_s$ of the SNC variety $\overline{D}_s=D_s\times_s\overline{s}$ is equal to  $\Delta({f^{-1}(x)})$. 
\end{enumerate}
\end{theorem}

\begin{proof} Let $\phi_0: X_0\to X$ be the Hironaka desingularization  of $X$ with SNC divisor $D_0$ and  denote by $\pi_0: X_0\to Y$  the induced morphism. By the assumption $\phi_0:X_0\to X$ is smooth over $\pi^{-1}(U)$. 
By generic smoothness there is a maximal open subset $U_1\subset Y$ (containing $U$) such that  $\pi_0^{-1}(U_1)\to U_1$ is smooth. Consider a  canonical Hironaka principalization  $\phi_1: X_1\to X_0$ of the ideal of the complement $Z_1:=X_0\setminus \pi_0^{-1}(U_1)$ on $(X_0,D_0)$. 

 Then  $D_1:=\phi_1^{-1}(Z_1)$ is a SNC divisor on $X_1$ with the induced morphism $\pi_1: X_1\to Y$. Set $Y_1:=Y\setminus U_1$ and consider the induced morphism $D_1\to Y_1$.
 Find a maximal nonsingular open subset $U_{1}\subset Y_{1}$ such that the morphism $\pi_1^{-1}(U_{1})\to U_{1}$ is SNC.  Set $Y_2:=Y_1\setminus U_{1}$, and take Hironaka principalization of  $Z_2:=\pi_1^{-1}(Y_2)$ on $(X_1,D_0\cup D_1)$.
 
 We continue this process until $Y_{k+1}=\emptyset$.  We obtain successive varieties $X_i$ with the exceptional SNC divisors $D_0\cup\ldots\cup D_i$ projections $\phi_i:X_i\to Y$, defining the restrictions $\pi_i: D_i\to Y_i$  and open subsets $U_i\subset Y_i$ for which $\pi_i^{-1}(U_{i})$ is SNC over $U_{i}$. We define stratification $S$ on $Y$ by taking irredducible components of the locally closed subsets $U_i\subset Y_i\subset Y$.
. 

\end{proof}
\bigskip
Applying the theorem to the identical morphism yields the corollary

\begin{theorem} {\bf Hironaka's equisingular  desingularization with SNC fibers}  Let $X$ be any algebraic variety over a field $K$, of characteristic zero.
There exists a canonical resolution, that is a nonsingular variety  $\overline{X}$ such that  $\overline{X}\to X$ has all SNC  fibers. Moreover

\begin{enumerate}
\item Let $U\subset X$ be the  open  subset of nonsingular points on $X$. 

\item The exceptional locus of $\phi$ is a  SNC divisor $D$.
\item The fibers of $\overline{\psi}$ are SNC varieties and  they have SNC with $D$.
\item There exist a nonsingular locally closed smooth stratification $S$ of $Y$ such that  the generic stratum is $U$ and for any nongeneric stratum $s\in S$ the set $\phi^{-1}(s)$ is a SNC divisor contained in $D$ and the morphism $\phi^{-1}(s)\to s$ is SNC.
\item The combinatorial type of the fibers is the same for all points in the stratum .
\end{enumerate}
\end{theorem}
\qed
\bigskip

\begin{theorem} {\bf  Hironaka's Canonical Principalization with SNC fibers} Let ${\cI}$ be a sheaf of ideals on a smooth algebraic variety $X$, and  $X\to Y$ be a proper morphism of varieties.
There exists a principalization of ${\cI}$  that is, a sequence

$$ X=X_0 \buildrel \sigma_1 \over\longleftarrow X_1
\buildrel \sigma_2 \over\longleftarrow X_2\longleftarrow\ldots
\longleftarrow X_i \longleftarrow\ldots \longleftarrow X_r =\widetilde{X}$$

of blow-ups $\sigma_i:X_{i-1}\leftarrow X_{i}$ of smooth centers $C_{i-1}\subset
 X_{i-1}$
such that

\begin{enumerate}

\item The exceptional divisor $E_i$ of the induced morphism $\sigma^i=\sigma_1\circ \ldots\circ\sigma_i:X_i\to X$ has only  simple normal
crossings and $C_i$ has simple normal crossings with $E_i$.

\item The
total transform $\sigma^{r*}({\cI})$ is the ideal of a simple normal
crossing divisor
$\widetilde{E}$ which is  a natural  combination of the irreducible components of the divisor ${E_r}$.
\item The fibers of $\overline{\psi}$ are SNC varieties and  they have SNC with $D$.
\item There exist a nonsingular locally closed stratification $S$ of $Y$ such that  the generic stratum is $U$ and for any nongeneric stratum $s\in S$ the set $\phi^{-1}(s)$ is a SNC divisor contained in $\widetilde{E}$ and the morphism $\phi^{-1}(s)\to s$ is SNC.
\item The combinatorial type of the fibers is the same for all points in the stratum.

\end{enumerate}
The morphism $(\widetilde{X},\widetilde{\cI})\rightarrow(X,{\cI}) $ defined by the above principalization  commutes with smooth morphisms and  embeddings of ambient varieties. It is equivariant with respect to any group action not necessarily preserving the ground field $K$.

\end{theorem}
\begin{proof} Consider the canonical Hironaka principalization $\sigma_1: X_1\to X$ of  the ideal of $\cI$. We create a SNC divisor $D_1$ on $Z_1$  with induced morphism $\pi_1: D_1\to Y$. Then find a maximal open subset $U_1$ of $Y$ such that $\pi_1^{-1}(U_1)\to U_1$ is strictly SNC  Corollary \ref{33}. Set $Y_2=Y\setminus U_1$ , and consider the canonical Hironaka principalization $\sigma_2: Z_2\to Z_1$ of $\cI_{Y_2}$ on $(Z_1,D_1)$. Then we create the excptional divisor $D_2=\sigma^{-1}_2(Y_2)$ having SNC crossing with the strict transform of $D_1$ and  definining the  morphism  $\pi_2:D_2\to Y_2$, which is SNC over an open  nonsingular $U_2\subset Y_2$. Put $Y_3:=Y_2\setminus U_2$ principalize $Y_3$ on $(Z_2, D_1\cup D_2)$, and continue the process untill $Y_{k+1}=\emptyset$. Then put $\overline{Z}:=Z_k$ and let $\overline{\pi}: \overline{D}:=D_1\cup\ldots\cup D_k\to Y$  $\overline{\sigma}: \overline{X}\to X$ be the induced projections.

Consider the stratification $S$ is determined by an irredducible components of the locally closed subsets $U_i\subset Y_i\subset Y$.
Then  $\overline{\pi}^{-1}(s)\to s$ is SNC and the combinatorial type  $\Delta(\overline{\pi}^{-1}(x))$ is the same for all $x\in s$.

\end{proof}

\begin{theorem} {\bf Constructibility of the combinatorial type of fibers}
Let $f: X\to Y$ be a projective morphism of complex projective varieties. Then there is a locally closed stratification $S=\{s\}$ of $Y$ such that for any stratum $s$ the combinatorial type of fibers $\Sigma(f^{-1}(x))$  is constant for all points $x\in s$. In particular the combinatorial type of fibers is a constructible invariant.
\end{theorem}

\begin{proof} $X, Y$ can be embedded as the closed subvarieties of $\bP^n$, and $\bP^m$ respectively. Then the graph of $f:X\to Y$ can be embedded 
into $\bP^n\times \bP^m$ so that we have the embedding $X\to \bP^n\times \bP^m$ defined by the graph and the commutative diagram

$$\begin{array}{rcl} 
X & \stackrel{f}{\to} & Y\\
 \dar & & \dar  \\
\bP^n\times \bP^m& \stackrel{\pi}{\to} &\bP^m
\end{array}$$

Consider the canonical Hironaka principalization $\overline{\sigma}: \overline{Z}\to \bP^n\times \bP^m$ of  the ideal of $X\to \bP^n\times \bP^m$ with respect to the morphism $\pi: \bP^n\times \bP^m\to \bP^m$.
Then the exceptional divisor $D$ maps to $Y$.

For any $x\in X$ the  fiber $f^{-1}(x)\subset X\subset \bP^n\times \bP^m$ we have that
$$\overline{\sigma}^{-1}(f^{-1}(x))=\overline{\pi}^{-1}(x),$$
which implies $$\Sigma(f^{-1}(x))=\Sigma(\overline{\sigma}^{-1}(f^{-1}(x)))=\Sigma(\overline{\pi}^{-1}(x))=(\Delta({\overline{\pi}^{-1}(x)}))$$
and is the same for all points in the same stratum $s$.

\end{proof}

\begin{corollary} {\bf Constructibility of zero Weight Deligne filtration of fibers}
Let $f: X\to Y$ be a projective morphism of complex projective varieties. Then there is a locally closed stratification $S=\{s\}$ of $Y$ such that for any stratum $s$ the zero weight Deligne filtratation $\dim W_0(H^i(f^{-1}(x))$ of the cohomology of the fiber $f^{-1}(x)$, is constant for all $x\in s$. In particular the invariant $\dim W_0H^i(f^{-1}(x))$ is constructible.
\end{corollary}

\end{document}